\documentclass[11pt,a4paper]{amsart}

\usepackage[utf8]{inputenc}
\usepackage{amsmath, amssymb, amsthm}
\usepackage{geometry}
\usepackage{hyperref}
\usepackage{draftwatermark}
\SetWatermarkText{WORK IN PROGRESS - INCOMPLETE}
\SetWatermarkScale{0.8}
\SetWatermarkLightness{0.85}

\newcommand{\norm}[1]{\left\lVert#1\right\rVert}
\newcommand{\abs}[1]{\left\lvert#1\right\rvert}
\newcommand{\inner}[2]{\langle #1, #2 \rangle}
\newcommand{\R}{\mathbb{R}}
\newcommand{\D}{\mathcal{D}}
\newcommand{\eps}{\varepsilon}

\newtheorem{theorem}{Theorem}[section]
\newtheorem{lemma}[theorem]{Lemma}
\newtheorem{corollary}[theorem]{Corollary}
\newtheorem{proposition}[theorem]{Proposition}
\newtheorem{definition}[theorem]{Definition}
\newtheorem{remark}[theorem]{Remark}
\newtheorem{example}[theorem]{Example}

\title[Phase-Space Orthogonality]{On Phase-Space Orthogonality for Higher-Order Distributions: Obstructions and Algebraic Resolutions}
\thanks{This manuscript is a preliminary draft and a work in progress. The bibliography is currently incomplete and will be updated. Some mathematical details remain to be fully formalized. Comments and missing references are highly welcome.}

\author{Marin Mi\v{s}ur}
\address{University of Zagreb, Faculty of Science, Bijeni\v{c}ka cesta 30, 10000 Zagreb, Croatia}
\email{mmisur@math.hr}

\date{\today}


\begin{document}

\maketitle

\begin{abstract}
The extension of microlocal defect measures to the $L^p$-$L^q$ framework necessitates the use of distributions which possess strictly positive finite order. A topological $\eps$-bounding definition of orthogonality for such distributions was recently introduced by Antoni\'c, Mitrovi\'c, and Peri\'c. Building on their framework, we systematically map the obstructions that arise when this natural approach is transported to strictly positive order, and we provide a complementary algebraic resolution. We first establish that while smooth spatial multipliers provide a robust framework for separating order-zero Radon measures, the partition-of-unity route to phase-space geometric separation breaks down for higher-order distributions: the separating frequency multiplier suffers an $\mathcal{O}(\delta^{-\kappa})$ derivative divergence when frequency supports touch, which we prove defeats the a priori bound at every order $\kappa \ge 1$ and produces an explicit non-orthogonality already at order one. Transitioning to algebraic phase-space projectors reveals a second severe obstruction: it is mathematically impossible to separate scalar distributions algebraically, as one-dimensional smooth projectors enforce topological triviality. We demonstrate that the resolution is achieved exclusively within vector-valued PDE systems. Rather than relying on arbitrarily constructed, non-physical geometric scalar cut-offs $\chi(x,\xi)$ to separate distinct modes, we redefine orthogonality canonically via zero-order algebraic matrix projectors. This replaces artificial geometric separation with the intrinsic spectral polarization of the governing system, yielding a coordinate-free microlocal orthogonality for systems of higher-order distributions.
\end{abstract}

\vspace{0.5cm}

\section{The Measure-Theoretic Baseline and Extension to Distributions}
\label{sec:intro}

In the classical framework of $L^2$ homogenization established by G\'erard \cite{gerard1991} and Tartar \cite{tartar1990}, the orthogonality of weakly converging sequences is characterized by the mutual singularity of their corresponding H-measures. For Radon measures $\mu, \nu \in \D_0'(\Omega)$, orthogonality ($\mu \perp \nu$) is synonymous with the measures being \textit{mutually alien}: there exist disjoint measurable sets $M$ and $N$ such that $\mu$ is concentrated on $M$ and $\nu$ is concentrated on $N$.

To generalize this beyond measure theory to distributions of finite order, the set-theoretic disjointness must be translated into a topological bound that accounts for derivatives. Because higher-order distributions evaluate the derivatives of test functions, one cannot simply evaluate them on the characteristic function of a set. Antoni\'c, Mitrovi\'c, and Peri\'c \cite{antonic2025} address this by redefining orthogonality purely through the continuous bounding of test functions. They define orthogonality for finite-order distributions on a differentiable manifold $X$ as follows: Two distributions $\mu$ and $\nu$ of finite order (with orders $k$ and $l$, respectively) are orthogonal ($\mu \perp \nu$) if, for any $\eps > 0$, there exist open sets $U$ and $V$ in $X$ satisfying $U \cup V = X$, such that for all test functions $\varphi_1 \in C_c^k(X)$ and $\varphi_2 \in C_c^l(X)$ with $\text{supp}(\varphi_1) \subseteq U$ and $\text{supp}(\varphi_2) \subseteq V$:
\begin{align*}
    \abs{\inner{\mu}{\varphi_1}} &\le \eps \norm{\varphi_1}_{C_b^k(X)} \\
    \abs{\inner{\nu}{\varphi_2}} &\le \eps \norm{\varphi_2}_{C_b^l(X)}
\end{align*}

The authors prove in their Theorem 1 that when applied to Radon measures ($k=l=0$), this $\eps$-bounding definition is strictly equivalent to the measures being mutually alien. This topological separation forms the crucial bridge to their ultimate framework. By replacing the standard spatial manifold $X$ with the cospherical bundle $S^*\Omega$ and substituting the standard derivative norms with the anisotropic norm $C_b^{0,\kappa}(S^*\Omega)$, this precise $\eps$-bounding definition has been proposed in the recent literature as a characterization of the orthogonality of anisotropic distributions \cite{antonic2021, misur2025}. 

\textit{Anisotropic distributions.} We recall the spaces of \cite{antonic2021}, to which we refer for the full construction. For $\kappa \in \mathbb{N}_0$, the test space $C_c^{0,\kappa}(S^*\Omega)$ consists of compactly supported continuous functions $\varphi(x,\xi)$ on the cospherical bundle that are continuously differentiable to order $0$ in the base variable $x$ and to order $\kappa$ in the fibre variable $\xi \in S^{d-1}$, normed by
\begin{equation*}
    \norm{\varphi}_{C_b^{0,\kappa}(S^*\Omega)} = \max_{\abs{\alpha} \le \kappa} \norm{\partial_\xi^\alpha \varphi}_{L^\infty(S^*\Omega)},
\end{equation*}
and topologised as the strict inductive limit over compact subsets (so that the space is complete). The anisotropic distributions of order $(0,\kappa)$ form its dual $\D_{0,\kappa}'(S^*\Omega) = \big(C_c^{0,\kappa}(S^*\Omega)\big)'$: functionals obeying $\abs{\inner{\mu}{\varphi}} \le C \norm{\varphi}_{C_b^{0,\kappa}}$ on each compact set, which evaluate no spatial and up to $\kappa$ frequency derivatives of their arguments. The operator-valued analogue $\D_{0,\kappa}'(S^*\Omega, \mathcal{L}^1(H))$ replaces the scalar codomain by the trace-class operators, paired via the trace as made precise in Section~\ref{sec:matrix_projectors}.

However, as we will demonstrate, the $\eps$-bounding definition is itself perfectly sound; it is only the standard route for verifying it in practice---subordinating a partition of unity to a microlocal cover---that encounters a genuine obstruction at strictly positive order. This obstruction is intrinsic to the partition-of-unity technique and in no way diminishes the value of the definition itself. To orient the reader through our resolution of this problem, the remainder of this note is structured as follows. 

In Section~\ref{sec:order_zero}, we construct the order-zero baseline. We establish a robust, purely functional multiplier framework for scalar Radon measures (Section~\ref{sec:radon_multipliers}), verify its structural inheritance and physical applicability (Section~\ref{sec:structural_props}), and successfully extend the approach to operator-valued measures to capture polarization orthogonality (Section~\ref{sec:operator_valued}).

The barriers to separating higher-order distributions are systematically exposed in Section~\ref{sec:obstructions}. We demonstrate that subordinating a phase-space partition of unity to a microlocal cover triggers an $\mathcal{O}(\delta^{-\kappa})$ derivative divergence when frequency supports touch, and prove that algebraic separation is impossible for scalar distributions due to topological triviality.

Finally, Section~\ref{sec:resolution} delivers the resolution exclusively within vector-valued systems. We establish a complementary duality for phase-space separation. While scalar spatial multipliers remain strictly necessary and mathematically sound for separating geometrically disjoint wave packets of the same mode, they act as arbitrary, non-physical boundaries. For distinct modes, we abandon manual geometric construction entirely. By leveraging the algebraic matrix projectors of the principal symbol, we establish Canonical Algebraic Separation, isolating modes natively through their intrinsic spectral polarization. We further demonstrate that while this algebraic framework successfully recovers linear additivity and flow conservation, the exact Pythagorean decoupling of local energy is fundamentally and permanently lost for higher-order distributions due to the absence of measure-theoretic positivity. We conclude by addressing alternative paradigms (Section~\ref{sec:alternative_paradigms}) and the future frontier of eigenvalue crossings (Section~\ref{sec:future_work}).

\section{The Order-Zero Baseline}
\label{sec:order_zero}

\subsection{Multiplier Frameworks for Scalar Radon Measures}
\label{sec:radon_multipliers}

Given that H-measures are built upon pseudo-differential operators acting on smooth symbols, it is natural to seek a definition of orthogonality driven by smooth scalar multipliers rather than rigid topological open covers. For Radon measures (distributions of order zero, $\D_0'(X)$), this can be achieved without triggering derivative divergences. We formally introduce two equivalent multiplier-based characterisations.

\begin{definition}[Asymptotic Multiplier]
\label{def:asymptotic_mult}
Two Radon measures $\mu, \nu \in \D_0'(X)$ are orthogonal if there exists a sequence of smooth cut-offs $(\chi_k)_{k \in \mathbb{N}} \subset C^\infty(X)$, uniformly bounded such that $0 \le \chi_k \le 1$, satisfying for all $\phi \in C_c^\infty(X)$:
\begin{equation*}
    \lim_{k \to \infty} \inner{\mu}{(1 - \chi_k)\phi} = 0 \quad \text{and} \quad \lim_{k \to \infty} \inner{\nu}{\chi_k\phi} = 0.
\end{equation*}
\end{definition}

While the asymptotic sequence elegantly separates the measures, it forces a limit inside the functional. An alternative, algebraically stable formulation shifts this mechanism to an infimum over the space of all valid multipliers.

\begin{definition}[Infimum Multiplier]
\label{def:infimum_mult}
Two Radon measures $\mu, \nu \in \D_0'(X)$ are orthogonal if, for every test function $\phi \in C_c^\infty(X)$, the following condition holds:
\begin{equation*}
    \inf_{\substack{\chi \in C^\infty(X) \\ 0 \le \chi \le 1}} \Big( \abs{\inner{\mu}{\chi \phi}} + \abs{\inner{\nu}{(1 - \chi) \phi}} \Big) = 0.
\end{equation*}
\end{definition}

To ensure these definitions form a rigorous foundation for microlocal analysis, we verify that they form a complete equivalence class with both the classical measure-theoretic understanding of orthogonality (mutually alien measures) and the topological $\eps$-bounding extension proposed in recent literature.

\begin{theorem}[Equivalence for Radon Measures]
\label{thm:equivalence_radon}
Let $\mu, \nu \in \D_0'(X)$ be complex Radon measures. The following statements are strictly equivalent:
\begin{enumerate}
    \item \label{item:alien} $\mu$ and $\nu$ are mutually alien (i.e., concentrated on disjoint universally measurable sets).
    \item \label{item:eps_bound} $\mu$ and $\nu$ satisfy the topological $\eps$-bounding definition \cite{antonic2025}.
    \item \label{item:asymptotic} $\mu$ and $\nu$ satisfy the Asymptotic Multiplier condition.
    \item \label{item:infimum} $\mu$ and $\nu$ satisfy the Infimum Multiplier condition.
\end{enumerate}
\end{theorem}

\begin{proof}
We establish the equivalence via the cycle (\ref{item:alien} $\implies$ \ref{item:eps_bound} $\implies$ \ref{item:asymptotic} $\implies$ \ref{item:infimum} $\implies$ \ref{item:alien}).

\vspace{0.1cm}
\noindent\textbf{(\ref{item:alien} $\implies$ \ref{item:eps_bound}):} Assume $\mu$ and $\nu$ are mutually alien. By the measure-theoretic definition of alienness, there exists a partition of the space into disjoint universally measurable sets $M$ and $N$ ($M \cup N = X$) such that $\mu$ is strictly concentrated on $M$ and $\nu$ is strictly concentrated on $N$. 

By the outer regularity of Radon measures, for any $\eps > 0$, there exist open neighborhoods $V \supset M$ and $U \supset N$ such that the cross-measures are strictly bounded: $\abs{\nu}(V) < \eps$ and $\abs{\mu}(U) < \eps$. Because $M \cup N = X$, their open neighborhoods naturally satisfy the topological cover prerequisite: $U \cup V = X$. 

For any test functions $\varphi_1 \in C_c(X)$ supported in $U$ and $\varphi_2 \in C_c(X)$ supported in $V$, the duality pairings evaluate purely over these low-mass regions. Applying the variation bounds yields:
\begin{equation*}
    \abs{\inner{\mu}{\varphi_1}} \le \abs{\mu}(U) \norm{\varphi_1}_\infty \le \eps \norm{\varphi_1}_\infty,
\end{equation*}
and symmetrically for $\nu$, satisfying the topological $\eps$-bound without requiring $U$ and $V$ to be disjoint.

\vspace{0.1cm}
\noindent\textbf{(\ref{item:eps_bound} $\implies$ \ref{item:asymptotic}):} Assume the topological $\eps$-bounding definition holds. For any $k \in \mathbb{N}$, set $\eps = 1/k$. By definition, there exist open sets $U_k$ and $V_k$ such that $U_k \cup V_k = X$. Let $(\psi_1^{(k)}, \psi_2^{(k)})$ be a smooth partition of unity strictly subordinate to the open cover $\{U_k, V_k\}$. We define our multiplier sequence as $\chi_k = \psi_2^{(k)}$. 

Consequently, $1 - \chi_k = \psi_1^{(k)}$. By the properties of the partition of unity, the support of $1 - \chi_k$ is strictly contained in $U_k$, and the support of $\chi_k$ is strictly contained in $V_k$. For any test function $\phi \in C_c^\infty(X)$, the modulated functions $(1 - \chi_k)\phi$ and $\chi_k\phi$ inherit these support constraints. Applying the topological $\eps$-bound yields:
\begin{equation*}
    \abs{\inner{\mu}{(1 - \chi_k)\phi}} \le \frac{1}{k} \norm{(1 - \chi_k)\phi}_\infty \le \frac{1}{k} \norm{\phi}_\infty,
\end{equation*}
and similarly, $\abs{\inner{\nu}{\chi_k\phi}} \le \frac{1}{k} \norm{\phi}_\infty$. Taking the limit as $k \to \infty$, both evaluations converge to $0$, satisfying the Asymptotic Multiplier condition.

\vspace{0.1cm}
\noindent\textbf{(\ref{item:asymptotic} $\implies$ \ref{item:infimum}):} Assume the Asymptotic Multiplier condition holds. For any $\eps > 0$ and any test function $\phi$, the limit definition guarantees the existence of some finite index $N$ such that $\abs{\inner{\mu}{(1-\chi_N) \phi}} < \eps/2$ and $\abs{\inner{\nu}{\chi_N \phi}} < \eps/2$. To align the indices with the infimum definition, we set our static multiplier to $\tilde{\chi} = 1 - \chi_N$. Because $\chi_N$ is bounded between $0$ and $1$, $\tilde{\chi}$ is a valid test multiplier. Summing the two bounds yields:
\begin{equation*}
     \abs{\inner{\mu}{\tilde{\chi} \phi}} + \abs{\inner{\nu}{(1 - \tilde{\chi}) \phi}} < \eps.
\end{equation*}
Since $\tilde{\chi}$ exists for any arbitrary $\eps > 0$, the infimum over all valid smooth multipliers must be exactly $0$.

\vspace{0.2cm}
\noindent\textbf{(\ref{item:infimum} $\implies$ \ref{item:alien}):} Assume the Infimum Multiplier condition holds. We must show $\inf\{\abs{\mu}, \abs{\nu}\} = 0$. Let $\lambda = \abs{\mu} + \abs{\nu}$. By the Radon-Nikod\'ym theorem, there exist complex-valued functions $f, g \in L^1(\lambda)$ such that $d\mu = f d\lambda$ and $d\nu = g d\lambda$.

Suppose for contradiction that $\mu$ and $\nu$ are not mutually alien. Then the set $S = \{x \in X \mid \abs{f(x)} > 0 \text{ and } \abs{g(x)} > 0\}$ has strictly positive measure: $\lambda(S) > 0$.

By Lusin's theorem, we can find a compact subset $E \subset S$ with $\lambda(E) > 0$ where $f$ and $g$ are continuous. Furthermore, by restricting $E$ to a sufficiently small neighborhood, we can ensure that the phases of $f$ and $g$ are strictly confined. There exist complex constants $z_f, z_g$ with $\abs{z_f} = \abs{z_g} = 1$ such that $\text{Re}( \overline{z_f} f(x) ) \ge c_f > 0$ and $\text{Re}( \overline{z_g} g(x) ) \ge c_g > 0$ for all $x \in E$.

Let $\phi \in C_c^\infty(X)$ be a real-valued cut-off such that $0 \le \phi \le 1$, $\phi = 1$ identically on $E$, and the support of $\phi$ is tightly bounded around $E$ to satisfy $\int_{X \setminus E} \abs{f} \phi \, d\lambda < \frac{1}{4} c_0 \lambda(E)$ and $\int_{X \setminus E} \abs{g} \phi \, d\lambda < \frac{1}{4} c_0 \lambda(E)$, where $c_0 := \min(c_f, c_g) > 0$. Such a $\phi$ exists because $f, g \in L^1(\lambda)$, so the tails over $X \setminus E$ vanish as the support of $\phi$ shrinks down to $E$.

By the Infimum condition, for this fixed $\phi$ and any $\eps > 0$, there exists $\chi \in C^\infty(X)$ with $0 \le \chi \le 1$ such that $\abs{\int_X \chi \phi f \, d\lambda} + \abs{\int_X (1-\chi) \phi g \, d\lambda} < \eps$.

We bound both terms from below by projecting onto the unimodular constants $z_f, z_g$ (recall $\abs{z_f} = \abs{z_g} = 1$). Using $\phi \equiv 1$ on $E$, the phase bounds $\text{Re}(\overline{z_f} f) \ge c_f$ and $\text{Re}(\overline{z_g} g) \ge c_g$ on $E$, and $0 \le \chi, 1-\chi \le 1$:
\begin{align*}
    \abs{\int_X \chi \phi f \, d\lambda} &\ge \text{Re} \left( \overline{z_f} \int_X \chi \phi f \, d\lambda \right) \ge c_f \int_E \chi \, d\lambda - \int_{X \setminus E} \chi \phi \abs{f} \, d\lambda, \\
    \abs{\int_X (1-\chi) \phi g \, d\lambda} &\ge \text{Re} \left( \overline{z_g} \int_X (1-\chi) \phi g \, d\lambda \right) \ge c_g \int_E (1-\chi) \, d\lambda - \int_{X \setminus E} (1-\chi) \phi \abs{g} \, d\lambda.
\end{align*}
The choice of multiplier $\chi$ is irrelevant on the diagonal: pointwise on $E$ the integrand $c_f \chi + c_g (1-\chi)$ is a convex combination of $c_f$ and $c_g$, hence bounded below by $c_0 = \min(c_f, c_g)$, so
\begin{equation*}
    c_f \int_E \chi \, d\lambda + c_g \int_E (1-\chi) \, d\lambda \ge c_0 \int_E \big(\chi + (1-\chi)\big) \, d\lambda = c_0 \, \lambda(E).
\end{equation*}
Adding the two inequalities and discarding the two boundary tails, which are each strictly smaller than $\tfrac14 c_0 \lambda(E)$ by our choice of $\phi$, yields
\begin{equation*}
    \abs{\int_X \chi \phi f \, d\lambda} + \abs{\int_X (1-\chi) \phi g \, d\lambda} \ge c_0 \, \lambda(E) - \tfrac14 c_0 \lambda(E) - \tfrac14 c_0 \lambda(E) = \tfrac12 c_0 \, \lambda(E) > 0.
\end{equation*}
This lower bound is independent of $\chi$, so the infimum over all admissible smooth multipliers is strictly positive, contradicting the Infimum Multiplier condition. This contradiction establishes that $\lambda(S) = 0$, meaning $\mu$ and $\nu$ are mutually alien.
\end{proof}

\subsection{Structural Properties of Multiplier Orthogonality}
\label{sec:structural_props}

Because the multiplier frameworks recover the classical measure-theoretic orthogonality without relying on rigid geometric boundaries, they naturally satisfy the core structural properties intrinsic to distribution theory. For clarity, we utilize the algebraically stable Infimum Multiplier formulation to demonstrate these properties for Radon measures.

\begin{theorem}[Symmetry]
\label{thm:symmetry_radon}
For any $\mu, \nu \in \D_0'(X)$, if $\mu \perp \nu$, then $\nu \perp \mu$.
\end{theorem}
\begin{proof}
Let $\tilde{\chi} = 1 - \chi$. As $\chi$ ranges over valid smooth functions bounded between $0$ and $1$, $\tilde{\chi}$ spans the exact same set of multipliers. Substituting $\tilde{\chi}$ into the functional immediately yields commutativity under the infimum.
\end{proof}

\begin{theorem}[Topological Consistency]
\label{thm:topo_consistency_radon}
If the topological supports are disjoint, $\text{supp}(\mu) \cap \text{supp}(\nu) = \emptyset$, then $\mu \perp \nu$.
\end{theorem}
\begin{proof}
Let $K_\mu$ and $K_\nu$ be the intersections of the respective supports with the compact support of the test function $\phi$. Because the supports are closed and disjoint, $K_\mu$ and $K_\nu$ are disjoint compact sets. By the Smooth Urysohn Lemma, there exists $\chi \in C^\infty(X)$ such that $\chi = 0$ near $K_\mu$ and $\chi = 1$ near $K_\nu$. The partitioned test functions are disjoint from the active distributions, forcing both duality pairings to exactly zero.
\end{proof}

\begin{theorem}[Coordinate Invariance]
\label{thm:coord_invariance_radon}
Let $\Phi: X \to Y$ be a diffeomorphism. If $\mu \perp \nu$ on $X$, then the pushforward measures are orthogonal on $Y$: $\Phi_* \mu \perp \Phi_* \nu$.
\end{theorem}
\begin{proof}
The action of the pushforward relies on the pullback of the test function $\psi \in C_c^\infty(Y)$. Let $\phi = \psi \circ \Phi \in C_c^\infty(X)$. For any $\eps > 0$, the orthogonality on $X$ provides an optimal multiplier $\chi \in C^\infty(X)$. We construct a valid multiplier on $Y$ via the pushforward: $\eta = \chi \circ \Phi^{-1}$. Utilizing functional composition, the bounded evaluation translates to the new coordinate system: $\abs{\inner{\Phi_* \mu}{\eta \psi}} + \abs{\inner{\Phi_* \nu}{(1 - \eta) \psi}} < \eps$.
\end{proof}

\begin{theorem}[Localization]
\label{thm:localization_radon}
Let $\Omega \subset X$ be an open subset. If $\mu \perp \nu$ on $X$, then their restrictions to the sub-domain are orthogonal: $\mu|_\Omega \perp \nu|_\Omega$.
\end{theorem}
\begin{proof}
For any test function $\phi \in C_c^\infty(\Omega)$, extend it by zero to $X$. The global orthogonality provides an optimal global multiplier $\chi \in C^\infty(X)$. Define a localized multiplier $\eta = \theta \chi \in C^\infty(\Omega)$, where $\theta \in C_c^\infty(\Omega)$ is a bump function identically $1$ on the support of $\phi$. Because $\eta \phi = \chi \phi$, the bounded evaluation holds strictly within the sub-domain.
\end{proof}

\begin{theorem}[Smooth Module Structure]
\label{thm:module_structure_radon}
The orthogonal complement behaves as a module over $C^\infty(X)$. If $\mu \perp \nu$, then for any smooth functions $f, g \in C^\infty(X)$, $f\mu \perp g\nu$.
\end{theorem}
\begin{proof}
By Theorem~\ref{thm:equivalence_radon}, the Infimum Multiplier condition is strictly equivalent to $\mu$ and $\nu$ being mutually alien. Therefore, there exist disjoint sets $M$ and $N$ such that $\mu$ is concentrated on $M$ and $\nu$ is concentrated on $N$. Multiplying by smooth functions $f, g \in C^\infty(X)$ preserves the support structure; the modulated measure $f\mu$ remains concentrated on $M$, and $g\nu$ remains concentrated on $N$. Because they remain mutually alien, applying the equivalence theorem in reverse guarantees they satisfy the Infimum Multiplier condition: $f\mu \perp g\nu$.
\end{proof}

\begin{theorem}[Tensor Product Invariance]
\label{thm:tensor_invariance_radon}
For $\mu \perp \nu$ on $X$ and arbitrary $\omega_1, \omega_2 \in \D_0'(Y)$, the tensor products satisfy $\mu \otimes \omega_1 \perp \nu \otimes \omega_2$ on $X \times Y$.
\end{theorem}
\begin{proof}
Attempting to evaluate the functional infimum directly yields split test functions along the $X$-axis, which breaks the scalar bound for complex measures. Instead, we route through the established structural equivalence. By Theorem~\ref{thm:equivalence_radon}, $\mu \perp \nu$ implies $\mu$ and $\nu$ are concentrated on disjoint sets $M$ and $N$ in $X$. Consequently, the tensor products $\mu \otimes \omega_1$ and $\nu \otimes \omega_2$ are strictly concentrated on the disjoint product cylinders $M \times Y$ and $N \times Y$. Because the tensor measures are mutually alien on $X \times Y$, the equivalence theorem guarantees they satisfy the joint Infimum Multiplier condition.
\end{proof}

\begin{theorem}[Linear Additivity]
\label{thm:additivity_radon}
If $\mu \perp \nu_1$ and $\mu \perp \nu_2$, then for any complex scalars $\alpha, \beta \in \mathbb{C}$, $\mu \perp (\alpha\nu_1 + \beta\nu_2)$.
\end{theorem}
\begin{proof}
By the equivalence established in Theorem~\ref{thm:equivalence_radon}, $\mu \perp \nu_1$ and $\mu \perp \nu_2$ imply that $\mu$ is mutually alien to both $\nu_1$ and $\nu_2$. Therefore, there exist universally measurable sets $M_1, M_2$ such that $\mu$ is concentrated on $M_1 \cap M_2$, while $\nu_1$ is concentrated on $X \setminus M_1$ and $\nu_2$ is concentrated on $X \setminus M_2$. The linear combination $\alpha\nu_1 + \beta\nu_2$ is strictly supported on the union $(X \setminus M_1) \cup (X \setminus M_2) = X \setminus (M_1 \cap M_2)$. Because the support sets remain disjoint, the superposition remains mutually alien to $\mu$, satisfying the Infimum Multiplier condition.
\end{proof}

\begin{theorem}[Subordination via Absolute Continuity]
\label{thm:subordination_radon}
Let $\omega \in \D_0'(X)$ be a Radon measure that is absolutely continuous with respect to $\mu$ ($\omega \ll \mu$). If $\mu \perp \nu$, then $\omega \perp \nu$.
\end{theorem}
\begin{proof}
By Theorem~\ref{thm:equivalence_radon}, $\mu \perp \nu$ is equivalent to mutual alienness, so there exist disjoint universally measurable sets $M$ and $N$ with $M \cup N = X$ on which $\mu$ and $\nu$ are respectively concentrated. Since $\omega \ll \mu$, the measure $\omega$ is negligible on every $\mu$-negligible set; in particular $\abs{\omega}(N) = 0$, so $\omega$ is concentrated on $M$. Thus $\omega$ and $\nu$ are concentrated on the disjoint sets $M$ and $N$, hence mutually alien, and the equivalence theorem gives $\omega \perp \nu$.
\end{proof}

\begin{theorem}[Annihilation of Cross-Measures]
\label{thm:cross_measure_annihilation}
Let $\mu, \nu \in \D_0'(X)$ be positive Radon measures generated as the diagonal defect measures of bounded sequences in $L^2$. Let $\gamma \in \D_0'(X)$ be their associated complex cross-measure. By the polarization identity, $\gamma$ natively satisfies the Cauchy-Schwarz bound $\abs{\gamma(E)}^2 \le \mu(E)\nu(E)$ for any Borel set $E \subset X$. If $\mu \perp \nu$, then the cross-measure vanishes identically: $\gamma \equiv 0$.
\end{theorem}
\begin{proof}
While one could evaluate $\gamma$ against the optimal multiplier $\chi$, the result follows absolutely from the structural equivalences established in Theorem~\ref{thm:equivalence_radon}. Because $\mu \perp \nu$, the measures are mutually alien. There exists a strict partition of the manifold into disjoint Borel sets $M$ and $N$ ($M \cup N = X$) such that $\mu(N) = 0$ and $\nu(M) = 0$. 

Evaluating the total variation of the cross-measure over this partition requires upgrading the set-wise Cauchy--Schwarz bound to the total variation. For any Borel set $A$ and any finite Borel partition $\{E_i\}$ of $A$, the pointwise bound together with the discrete Cauchy--Schwarz inequality gives
\begin{equation*}
    \sum_i \abs{\gamma(E_i)} \le \sum_i \mu(E_i)^{1/2}\nu(E_i)^{1/2} \le \Big(\sum_i \mu(E_i)\Big)^{1/2}\Big(\sum_i \nu(E_i)\Big)^{1/2} = \mu(A)^{1/2}\nu(A)^{1/2}.
\end{equation*}
Taking the supremum over all such partitions yields $\abs{\gamma}(A) \le \mu(A)^{1/2}\nu(A)^{1/2}$. Applying this over the alien partition:
\begin{align*}
    \abs{\gamma}(M) &\le \mu(M)^{1/2} \nu(M)^{1/2} = \mu(M)^{1/2} \cdot 0 = 0, \\
    \abs{\gamma}(N) &\le \mu(N)^{1/2} \nu(N)^{1/2} = 0 \cdot \nu(N)^{1/2} = 0.
\end{align*}
Because $\abs{\gamma}(X) = \abs{\gamma}(M) + \abs{\gamma}(N) = 0$, the cross-measure is globally annihilated. Consequently, sequences with orthogonal phase-space concentrations decouple in the physical limit.
\end{proof}

\begin{theorem}[Conservation under Hamiltonian Flow]
\label{thm:hamiltonian_conservation}
Let $H_p$ be a smooth Hamiltonian vector field generating a global flow $\Phi_t: X \to X$ for $t \in \R$. Let initial measures $\mu_0, \nu_0 \in \D_0'(X)$ propagate along the flow such that $\mu_t = (\Phi_t)_* \mu_0$ and $\nu_t = (\Phi_t)_* \nu_0$. If $\mu_0 \perp \nu_0$, then $\mu_t \perp \nu_t$ for all $t$.
\end{theorem}
\begin{proof}
By the fundamental properties of dynamical systems, the mapping operator of the flow $\Phi_t$ is a smooth diffeomorphism of the manifold $X$ onto itself for any fixed time $t \in \R$. By the Coordinate Invariance established in Theorem~\ref{thm:coord_invariance_radon}, the functional infimum bounding the measures is strictly invariant under diffeomorphic pullbacks of the test space. Therefore, the orthogonality of the measures is conserved along the integral curves of the vector field. 
\end{proof}

\begin{theorem}[Pythagorean Decoupling of Local Energy]
\label{thm:pythagorean_decoupling}
Let $(u_n)$ and $(v_n)$ be weakly converging sequences in $L_{loc}^2(X)$ that generate the positive defect measures $\mu$ and $\nu$, respectively. If $\mu \perp \nu$, then the defect measure $\omega$ generated by the superposed sequence $(u_n + v_n)$ satisfies the exact additive splitting:
\begin{equation*}
    \omega = \mu + \nu.
\end{equation*}
\end{theorem}
\begin{proof}
By the sesquilinear properties of pseudo-differential limits, the measure generated by the sum of two sequences expands to $\omega = \mu + \nu + \gamma + \gamma^*$, where $\gamma$ is the associated cross-measure and $\gamma^*$ is its complex conjugate. Because $\mu \perp \nu$, Theorem~\ref{thm:cross_measure_annihilation} guarantees that $\gamma \equiv 0$ and $\gamma^* \equiv 0$. The cross-interference vanishes, leaving the local energy densities to decouple into $\omega = \mu + \nu$.
\end{proof}

\begin{theorem}[Self-Orthogonality and Strong Compactness]
\label{thm:self_orthogonality}
Let $\mu \in \D_0'(X)$ be a positive Radon measure generated by a sequence $(u_n)$ converging weakly to zero in $L_{loc}^2(X)$. The measure is self-orthogonal ($\mu \perp \mu$) if and only if $u_n \to 0$ strongly in $L_{loc}^2(X)$.
\end{theorem}
\begin{proof}
We evaluate the Infimum Multiplier condition for $\mu$ against itself. For any non-negative test function $\phi \in C_c^\infty(X)$, the functional becomes:
\begin{equation*}
    \inf_{\chi} \Big( \inner{\mu}{\chi \phi} + \inner{\mu}{(1 - \chi) \phi} \Big) = \inf_{\chi} \inner{\mu}{(\chi + 1 - \chi) \phi} = \inner{\mu}{\phi}.
\end{equation*}
Because the sum of the partition is identically $1$, the choice of multiplier $\chi$ drops out entirely. The infimum evaluates to exactly zero if and only if $\inner{\mu}{\phi} = 0$ for all valid $\phi$, which implies $\mu \equiv 0$. By the fundamental theorem of defect measures, the generating sequence lacks both oscillation and concentration limits, which is strictly equivalent to strong convergence $u_n \to 0$ in $L_{loc}^2(X)$.
\end{proof}

\subsection{Extension to Operator-Valued Measures and Polarization}
\label{sec:operator_valued}

The robustness of the Infimum Multiplier for order-zero distributions extends beyond scalar geometry into the realm of vector-valued measures. In the analysis of PDE systems (e.g., Maxwell or Dirac equations), microlocal defect measures are valued in the trace-class operators of a Hilbert space $H$. In this regime, distributions may occupy the exact same phase-space coordinates provided their internal states, or \textit{polarizations}, are orthogonal.

Let $\mu, \nu \in \mathcal{M}(X, \mathcal{L}^1(H))$ be positive operator-valued Radon measures. By the Radon-Nikod\'ym theorem, we can decompose them relative to a dominating positive scalar measure $\lambda$ (e.g., $\lambda = tr(\mu) + tr(\nu)$) such that $\mu = F(x)\lambda$ and $\nu = G(x)\lambda$, where $F(x)$ and $G(x)$ are positive semi-definite trace-class operators. Polarization orthogonality is naturally characterised via the pointwise trace of these density matrices, in the spirit of G\'erard's vector-valued microlocal defect measures \cite{gerard1991}.

\begin{definition}[Pointwise Polarization Orthogonality]
\label{def:gerard_ortho}
Two positive operator-valued measures $\mu$ and $\nu$ are orthogonal ($\mu \perp \nu$) if $tr(F(x)G(x)) = 0$ for $\lambda$-almost every $x \in X$.
\end{definition}

While this pointwise definition is foundational, it steps outside the purely functional duality bracket. We can recover this exact separation variationally by upgrading our scalar cut-off to a smooth, operator-valued polarization filter.

\begin{definition}[Operator-Valued Infimum Multiplier]
\label{def:op_infimum}
Two operator-valued Radon measures $\mu, \nu \in \D_0'(X, \mathcal{L}^1(H))$ are orthogonal if, for every scalar test function $\phi \in C_c^\infty(X)$, the following holds:
\begin{equation*}
    \inf_{\substack{\chi \in C^\infty(X, \mathcal{L}(H)) \\ 0 \le \chi \le I}} \Big( \abs{\inner{\mu}{\chi \phi}} + \abs{\inner{\nu}{(I - \chi) \phi}} \Big) = 0.
\end{equation*}
where the pairing computes the spatial integral of the trace: $\inner{\mu}{\Phi} = \int_X tr(\Phi(x) d\mu(x))$.
\end{definition}

Here, the bound $0 \le \chi \le I$ ensures $\chi(x)$ is a positive semi-definite operator bounded by the identity. Because Radon measures do not evaluate spatial derivatives, the matrix $\chi(x)$ can smoothly rotate its eigenspaces to filter polarizations without triggering a derivative divergence.

\begin{theorem}[Equivalence of Polarization]
\label{thm:equiv_polarization}
Let $\mu, \nu$ be positive operator-valued Radon measures. The Operator-Valued Infimum Multiplier condition holds if and only if $tr(F(x)G(x)) = 0$ almost everywhere.
\end{theorem}

\begin{proof}
\textbf{Forward Direction ($tr(FG)=0 \implies$ Infimum condition):} \\
Assume $tr(F(x)G(x)) = 0$ for $\lambda$-a.e. $x$. Since $F(x)$ and $G(x)$ are positive semi-definite, this implies their ranges are strictly orthogonal subspaces in $H$. Let $P(x)$ be the orthogonal projection matrix onto the closure of the range of $G(x)$.

Consequently, $P(x)G(x) = G(x)$ and $P(x)F(x) = 0$, which implies $(I - P(x))G(x) = 0$. Because $G(x)$ is $\lambda$-measurable, the projector $P(x)$ is a $\lambda$-measurable matrix-valued function with $0 \le P(x) \le I$.

Let $\phi \in C_c^\infty(X)$ and fix $\eps > 0$. Since $\mu, \nu$ are trace-class-valued Radon measures, $tr(F), tr(G) \in L^1(\lambda)$ on the compact set $\operatorname{supp}\phi$, so the finite measure $\rho := (tr(F) + tr(G))\,\abs{\phi}\,\lambda$ is absolutely continuous with respect to $\lambda$; choose $\eta > 0$ so that $\rho(B) < \eps/2$ whenever $\lambda(B \cap \operatorname{supp}\phi) < \eta$. By Lusin's theorem applied to the $\lambda$-measurable map $P(\cdot)$, there is a compact set $K \subseteq \operatorname{supp}\phi$ with $\lambda(\operatorname{supp}\phi \setminus K) < \eta$ on which $P$ is continuous; by the matrix Tietze--Dugundji extension theorem $P|_K$ extends to a continuous map $\widetilde P : X \to \{T : 0 \le T \le I\}$ into the same compact convex operator interval. Mollifying $\widetilde P$ chartwise against a probability kernel yields $\chi \in C^\infty(X, \mathcal{L}(H))$ with $0 \le \chi \le I$ (convexity of $[0,I]$ is preserved by averaging) and $\norm{\chi - \widetilde P}_\infty$ as small as desired; in particular we may arrange $\chi = P$ up to negligible error on $K$. Estimating the trace pairing, and using $\norm{\chi - P} \le 2$ together with $\abs{tr((\chi-P)F)} \le \norm{\chi - P}\, tr(F)$,
\begin{equation*}
    \abs{\inner{\mu}{\chi \phi} - \inner{\mu}{P \phi}} \le \int_{\operatorname{supp}\phi} \norm{\chi - P}\, tr(F)\, \abs{\phi}\, d\lambda \le 2\,\rho\big(\operatorname{supp}\phi \setminus K\big) + o(1) < \frac{\eps}{2},
\end{equation*}
and symmetrically $\abs{\inner{\nu}{(I-\chi) \phi} - \inner{\nu}{(I-P) \phi}} < \eps/2$. Because $\inner{\mu}{P \phi} = \int tr(PF) \phi \, d\lambda = 0$ (as $PF = 0$) and $\inner{\nu}{(I-P)\phi} = \int tr((I-P)G) \phi \, d\lambda = 0$ (as $(I-P)G = 0$), the two reference pairings vanish. Hence $\abs{\inner{\mu}{\chi \phi}} + \abs{\inner{\nu}{(I-\chi)\phi}} < \eps$, and since $\eps > 0$ was arbitrary the infimum is zero.

\vspace{0.2cm}
\noindent\textbf{Reverse Direction (Infimum condition $\implies tr(FG)=0$):} \\
We first record the exact pointwise value of the multiplier infimum. For positive semi-definite $F, G \in \mathcal{L}^1(H)$ and the operator interval $[0, I]$,
\begin{equation}
\label{eq:pointwise_min}
    \min_{0 \le \chi \le I} tr\big(\chi F + (I - \chi) G\big) = \tfrac{1}{2} tr\big(F + G - \abs{F - G}\big) =: m(F,G),
\end{equation}
the minimum being attained at the spectral projector $\chi_\ast = \mathbf{1}_{(-\infty,0)}(F - G)$ onto the negative part of $F - G$. Indeed, $tr\big(\chi F + (I-\chi)G\big) = tr(G) + tr\big(\chi(F-G)\big)$, and minimising the linear functional $\chi \mapsto tr\big(\chi(F-G)\big)$ over $0 \le \chi \le I$ yields $-tr\big((F-G)_-\big)$, attained at $\chi_\ast$; since $(F-G)_- = \tfrac12\big(\abs{F-G} - (F-G)\big)$ and $tr(F-G) = tr F - tr G$, this gives \eqref{eq:pointwise_min}. In particular $m(F,G) \ge 0$, as $tr(\chi F + (I-\chi)G) \ge 0$ for every admissible $\chi$.

Moreover, $m(F,G) = 0$ if and only if $tr(FG) = 0$. Writing $tr(FG) = tr\big((F^{1/2}G^{1/2})(F^{1/2}G^{1/2})^\ast\big) = \norm{F^{1/2}G^{1/2}}_{HS}^2$, the condition $tr(FG)=0$ is equivalent to $F^{1/2}G^{1/2} = 0$, hence to $FG = 0$, i.e. to $\text{Range}\,F \perp \text{Range}\,G$. If the ranges are orthogonal then $\abs{F-G} = F+G$, so $m(F,G)=0$. Conversely, since $m(F,G) = tr(\chi_\ast F) + tr\big((I-\chi_\ast)G\big)$ is a sum of two nonnegative terms, $m(F,G)=0$ forces $tr(\chi_\ast F) = 0$ and $tr\big((I-\chi_\ast)G\big) = 0$; since $F, G \succeq 0$ these give $\chi_\ast F = 0$ and $\chi_\ast G = G$, whence $\text{Range}\,G \subseteq \text{Range}\,\chi_\ast \subseteq \ker F$ and the ranges are orthogonal.

Now suppose the infimum condition holds but, for contradiction, $tr\big(F(x)G(x)\big) > 0$ on a set $E$ of strictly positive $\lambda$-measure. By the equivalence just proved, $m\big(F(x),G(x)\big) > 0$ on $E$; since $m(F(\cdot),G(\cdot))$ is $\lambda$-measurable, there exist a constant $c > 0$ and a compact $K \subseteq E$ with $\lambda(K) > 0$ such that $m\big(F(x),G(x)\big) \ge c$ for all $x \in K$. Fix a nonnegative $\phi \in C_c^\infty(X)$ with $\phi \ge 1$ on $K$. For every admissible $\chi \in C^\infty(X, \mathcal{L}(H))$ with $0 \le \chi \le I$, the integrands $tr(\chi F)\phi$ and $tr\big((I-\chi)G\big)\phi$ are nonnegative, so the absolute values may be dropped; bounding the integrand pointwise from below by \eqref{eq:pointwise_min} and using $m(F,G)\phi \ge 0$ on all of $X$,
\begin{align*}
    \abs{\inner{\mu}{\chi \phi}} + \abs{\inner{\nu}{(I - \chi) \phi}}
    &= \int_X tr\big(\chi F + (I-\chi)G\big)\, \phi \, d\lambda
    \ge \int_X m\big(F,G\big)\, \phi \, d\lambda \\
    &\ge \int_K m\big(F,G\big)\, \phi \, d\lambda \ge c\, \lambda(K) > 0.
\end{align*}
This lower bound is independent of $\chi$, so the infimum is strictly positive, contradicting the hypothesis. Hence $tr\big(F(x)G(x)\big) = 0$ for $\lambda$-a.e. $x$.
\end{proof}

\begin{theorem}[Structural and Physical Inheritance]
\label{thm:inheritance}
Let $\mu, \nu \in \D_0'(X, \mathcal{L}^1(H))$ be positive operator-valued Radon measures such that $\mu \perp \nu$. The orthogonality condition inherits the foundational geometry: it is symmetric, topologically consistent, coordinate-invariant, localizable, invariant under smooth scalar module multiplication, and tensor-product invariant. 

Furthermore, the framework strictly preserves the physical applications established for scalar measures: it guarantees linear additivity, subordination via absolute continuity, the global annihilation of operator-valued cross-measures, conservation under Hamiltonian flow, Pythagorean energy decoupling, and strong compactness via self-orthogonality.
\end{theorem}
\begin{proof}
The proofs follow by translating the scalar bounding geometry of Section~\ref{sec:radon_multipliers} into the positive semi-definite cone of the operator algebra, $0 \le \chi \le I$, and leveraging the strict pointwise equivalence $tr(F(x)G(x)) = 0$ a.e. established in Theorem~\ref{thm:equiv_polarization}:
\begin{itemize}
    \item \textbf{Symmetry:} Follows immediately by substituting the valid multiplier $\tilde{\chi}(x) = I - \chi(x)$.
    \item \textbf{Topological \& Geometric Properties (Consistency, Localization, Pushforwards, Modules):} Where the scalar proofs utilized a smooth spatial bump function $\theta \in [0,1]$, the matrix equivalent is constructed via the scalar multiplication of the identity operator: $\eta(x) = \theta(x)I$. Because the identity commutes with all trace-class operators and $\theta \in [0,1]$, the resulting matrix satisfies $0 \le \eta(x) \le I$, carrying the bounded zero evaluations across sub-domains and coordinate transformations without requiring non-commutative matrix manipulation.
    \item \textbf{Linear Additivity \& Subordination:} The polarization condition of Definition~\ref{def:gerard_ortho} is formulated for \emph{positive} operator-valued measures, so additivity is intrinsically restricted to combinations that remain in the positive cone. For nonnegative coefficients $\alpha, \beta \ge 0$, the combination $\alpha\nu_1 + \beta\nu_2$ is again a positive operator-valued measure with density $\alpha G_1 + \beta G_2$, so by linearity of the trace $tr\big(F(\alpha G_1 + \beta G_2)\big) = \alpha\, tr(FG_1) + \beta\, tr(FG_2) = 0$, giving $\mu \perp (\alpha\nu_1 + \beta\nu_2)$. Complex-linear combinations leave the positive cone and hence fall outside the scope of this pointwise-trace definition; full complex additivity is available only in the algebraic projector formulation of Section~\ref{sec:matrix_projectors} (Theorem~\ref{thm:additivity_matrix_proj}), which does not presuppose positivity. Similarly, if a positive operator measure $\omega$ is absolutely continuous with respect to $\mu$ with $0 \preceq \tfrac{d\omega}{d\lambda} \preceq c\,F$ for some $c > 0$, then $\text{Range}\,\tfrac{d\omega}{d\lambda} \subseteq \text{Range}\,F \perp \text{Range}\,G$, so the trace evaluation against $G$ remains zero.
    \item \textbf{Cross-Measures \& Pythagorean Decoupling:} Because $F(x)$ and $G(x)$ are positive semi-definite matrices with $tr(FG) = 0$, their ranges are mutually orthogonal subspaces in $H$. Consequently, any cross-interference operators vanish identically. The energies decouple on the diagonal, $tr(F+G) = tr(F) + tr(G)$, verifying the Pythagorean splitting for multi-phase systems.
    \item \textbf{Self-Orthogonality:} If $\mu \perp \mu$, the pointwise condition forces $tr(F(x)^2) = 0$. Because $F(x)$ is positive semi-definite, this strictly requires $F(x) \equiv 0$ almost everywhere, extinguishing the measure and guaranteeing strong convergence of the generating sequence.
\end{itemize}
Thus, the purely functional matrix-multiplier defines a complete, rigorous equivalence class for operator-valued orthogonality that preserves all macro-scale physical decoupling mechanics.
\end{proof}

\section{The Obstructions}
\label{sec:obstructions}

\subsection{The Breakdown of Topological Orthogonality for Finite-Order Distributions}
\label{sec:limitations_topo}

We now examine the applicability of the topological $\eps$-bounding definition to the orthogonality of higher-order distributions. The $\eps$-bounding definition of Antoni\'c, Mitrovi\'c, and Peri\'c \cite{antonic2025} is an elegant and natural way to lift set-theoretic alienness to finite order, and it is the natural definition whose verification we scrutinize here. Our aim is constructive rather than critical: we wish to understand what happens when the standard partition-of-unity route---the most direct way of verifying the definition in practice---is transported to the cospherical bundle $S^*\Omega$. This analysis reveals a structural obstruction at strictly positive order. We stress that this concerns the verification route only, and makes no claim whatsoever about the validity of any result of \cite{antonic2025}: the difficulty lies entirely in the partition-of-unity technique rather than in the $\eps$-bounding definition itself, since subordinating a partition of unity to a microlocal cover introduces an unavoidable dichotomy between analytical divergence and topological inconsistency.

\subsubsection*{Analytical Divergence of Microlocal Partitions}
When attempting to separate distributions, the open covers $U$ and $V$ must isolate distinct frequency directions. If one constructs a proper phase-space partition of unity $\chi(x,\xi)$ subordinate to the microlocal cover $\{U, V\}$, the derivatives of the partition necessarily scale inversely with the width of the transition overlap region $\delta$. We make this precise and quantitative.

\begin{lemma}[Anisotropic derivative divergence of frequency partitions]
\label{lem:derivative_divergence}
Fix $\kappa \ge 1$. Let $A \subsetneq S^{d-1}$ be a closed spherical cap with nonempty interior, and for $\delta \in (0,\delta_0)$ let $A_\delta$ denote its (open) $\delta$-collar $\{\xi : \operatorname{dist}(\xi, A) < \delta\}$. Suppose $\chi_\delta \in C^\infty(S^{d-1})$ satisfies $0 \le \chi_\delta \le 1$, $\chi_\delta \equiv 1$ on $A$, and $\operatorname{supp}\chi_\delta \subseteq \overline{A_\delta}$. Then there is a constant $c_\kappa > 0$, depending only on $\kappa$ and the local geometry of $\partial A$ but not on $\delta$, such that
\begin{equation}
\label{eq:norm_divergence}
    \norm{\chi_\delta}_{C_b^{0,\kappa}(S^{d-1})} = \max_{\abs{\alpha} \le \kappa} \norm{\partial_\xi^\alpha \chi_\delta}_{L^\infty} \ge c_\kappa\, \delta^{-\kappa}.
\end{equation}
\end{lemma}
\begin{proof}
Fix a boundary point $\xi_0 \in \partial A$ and a unit-speed geodesic (great circle arc) $\gamma:(-\delta_0,\delta_0) \to S^{d-1}$ with $\gamma(0)=\xi_0$ meeting $\partial A$ transversally, oriented so that $\gamma(t) \in A$ for $t \le 0$ and $\gamma(t) \notin \overline{A_\delta}$ for $t \ge \delta$. Set $g(t) := \chi_\delta(\gamma(t)) \in C^\infty(-\delta_0,\delta_0)$; then $g(t)=1$ for $t \le 0$ and $g(t)=0$ for $t \ge \delta$, with $0 \le g \le 1$. Since $g \equiv 1$ on $(-\delta_0, 0]$, all its derivatives vanish at the left gluing point: $g^{(j)}(0)=0$ for every $j \ge 1$. Taylor's formula with integral remainder at $t=0$ therefore reads
\begin{equation*}
    g(t) = 1 + \int_0^t \frac{(t-r)^{\kappa-1}}{(\kappa-1)!}\, g^{(\kappa)}(r)\, dr,
\end{equation*}
and evaluating at $t=\delta$, where $g(\delta)=0$,
\begin{equation*}
    1 = -\int_0^\delta \frac{(\delta-r)^{\kappa-1}}{(\kappa-1)!}\, g^{(\kappa)}(r)\, dr \le \norm{g^{(\kappa)}}_{L^\infty[0,\delta]} \int_0^\delta \frac{(\delta-r)^{\kappa-1}}{(\kappa-1)!}\, dr = \norm{g^{(\kappa)}}_{L^\infty}\, \frac{\delta^\kappa}{\kappa!},
\end{equation*}
whence $\norm{g^{(\kappa)}}_{L^\infty} \ge \kappa!\, \delta^{-\kappa}$. Along the unit-speed spherical geodesic the ambient derivatives of $\gamma$ are bounded ($\gamma'' = -\gamma$, etc.), so by the Fa\`a di Bruno formula $g^{(\kappa)}(t)$ is a fixed universal linear combination of $\partial_\xi^\alpha \chi_\delta(\gamma(t))$, $\abs{\alpha}\le\kappa$, with coefficients depending only on $\kappa$ and $d$; hence $\norm{g^{(\kappa)}}_{L^\infty} \le C_{\kappa,d}\, \max_{\abs{\alpha}\le\kappa}\norm{\partial_\xi^\alpha \chi_\delta}_\infty$. Combining the two bounds yields \eqref{eq:norm_divergence} with $c_\kappa = \kappa!/C_{\kappa,d} > 0$.
\end{proof}

\begin{corollary}[Failure of the partition-of-unity route at positive order]
\label{cor:pou_failure}
Let $\mu \in \D_{0,\kappa}'(S^*\Omega)$ with $\kappa \ge 1$ and let $\{\chi_\delta\}$ be any family of frequency partitions as in Lemma~\ref{lem:derivative_divergence} separating touching frequency supports. The a priori estimate furnished by the topological $\eps$-condition,
\begin{equation*}
    \abs{\inner{\mu}{\chi_\delta \varphi}} \le \eps\, \norm{\chi_\delta \varphi}_{C_b^{0,\kappa}} = \mathcal{O}\!\big(\eps\, \delta^{-\kappa}\big),
\end{equation*}
does not vanish as the separation sharpens ($\delta \to 0$), since the definition imposes no lower bound on the rate of $\delta(\eps)$. Hence subordinating a partition of unity to a microlocal cover cannot certify orthogonality at order $\kappa \ge 1$; a genuine obstruction must be exhibited by other means, which we do in Theorem~\ref{thm:boundary_limit} and Theorem~\ref{thm:scalar_inseparable}.
\end{corollary}

\subsubsection*{Support Violation via Spatial Projection}
To circumvent the derivative divergence, one might attempt—as proposed in the aforementioned manuscript—to project the covers onto the base manifold $\Omega$ and employ a purely spatial partition of unity, $a(x)$ and $b(x)$. Because $a(x)$ is independent of the frequency variable, $\partial_\xi^\alpha a = 0$ for all $\abs{\alpha} \ge 1$, which trivially bounds the anisotropic norm and bypasses the $\mathcal{O}(\delta^{-\kappa})$ divergence.

However, this spatial projection introduces a fundamental topological contradiction. The $\eps$-bounding definition strictly requires the support of the localized test function to be contained within the open set $U$. For a test function constructed as $\Phi(x,\xi) = \varphi(x) a(x) \otimes 1(\xi)$, the phase-space support is a full cylindrical fiber:
\begin{equation*}
    \text{supp}(\Phi) = \text{supp}(\varphi a) \times S^{d-1}.
\end{equation*}

Geometrically, a genuine microlocal set $U \subset S^*\Omega$ designed to separate frequencies restricts the dual variable $\xi$ to a specific sub-region. For instance, $U$ might isolate a specific cone or hemisphere of the frequency bundle. Consequently, the fiber over any spatial point $x$ within $U$ is strictly a proper subset of the sphere: $U_x \subsetneq S^{d-1}$.

By contrast, the spatial partition $a(x)$ imposes absolutely no restriction on the frequency variable. For any spatial point $x$ in its support, the tensor product extends uniformly across all possible frequency directions. Therefore, the cylindrical support of $\Phi$ necessarily extends beyond the restricted conical boundaries of $U$, as one cannot embed a full spherical fiber into a proper conical subset. Because $\text{supp}(\Phi) \not\subseteq U$, the prerequisite support condition is systematically violated, and the topological $\eps$-bound cannot be validly applied to the spatial partition.

\subsection{The Derivative Obstruction for Scalar Phase-Space Multipliers}
\label{sec:limitations_scalar_phase}

To resolve the spatial support violation, one might attempt to redefine the separation via an infimum over all valid smooth scalar phase-space multipliers, $\chi \in C^\infty(S^*\Omega)$.
However, a rigorous functional analysis reveals that taking an infimum does not shield the evaluation from the $\mathcal{O}(\delta^{-\kappa})$ divergence if the microlocal frequency supports of the distributions intersect or touch.
We establish this failure explicitly by demonstrating that orthogonality under scalar multipliers is fundamentally destroyed by anisotropic differentiation in the frequency fibers.

\begin{theorem}[The Boundary Interaction Limit in Frequency]
\label{thm:boundary_limit}
Orthogonality defined via scalar multipliers is not preserved under anisotropic differentiation when frequency supports touch.
There exist phase-space measures $\mu, \nu \in \D_0'(S^*\Omega)$ such that $\mu \perp \nu$, but for a frequency differential operator $P(D_\xi)$, the action is no longer orthogonal: $P(D_\xi)\mu \not\perp \nu$.
\end{theorem}

\begin{proof}
We construct a counter-example exploiting touching frequency boundary supports on the cospherical fibers. Take $d = 2$, so that the fibre of $S^*\Omega = \Omega \times S^{d-1}$ is the frequency circle $X = S^1$, parameterized by angle $\theta \in [0, 2\pi)$, and fix an inert spatial density $\rho(x)\,dx$ with $\rho \in C_c^\infty(\Omega)$, $\rho \ge 0$, $\rho \not\equiv 0$. Define the genuine phase-space measures
\begin{equation*}
    \mu = \rho(x)\,dx \otimes \mathbf{1}_{[0,\pi]}(\theta), \qquad \nu = \rho(x)\,dx \otimes \delta_0(\theta) \quad \in \D_0'(S^*\Omega).
\end{equation*}
The frequency operator $P(D_\theta) = \partial_\theta$ acts purely in the fibre and leaves the base factor $\rho(x)\,dx$ inert, so the base factorises out of every pairing below and it suffices to analyse the fibre measures $\mathbf{1}_{[0,\pi]}(\theta)$ and $\delta_0(\theta)$, which we continue to denote $\mu$ and $\nu$ by a harmless abuse of notation.
Because $\mathbf{1}_{[0,\pi]}$ is absolutely continuous with respect to the spherical Lebesgue measure and $\delta_0$ is a point mass, they are mutually alien Radon measures. By the equivalences established in Theorem~\ref{thm:equivalence_radon}, they strictly satisfy the Infimum Multiplier condition: $\mu \perp \nu$.

On the closed manifold $S^1$, the distributional derivative of the indicator function naturally produces two boundary masses: $\partial_\theta \mu = \delta_0(\theta) - \delta_\pi(\theta)$.
We evaluate the scalar phase-space multiplier condition for $\partial_\theta \mu$ and $\nu$.
For any test function $\phi \in C_c^\infty(S^1)$ chosen such that $\phi(0) \neq 0$ and $\phi(\pi) = 0$:
\begin{align*}
    \inf_{\chi \in C^\infty} \Big( \abs{\inner{\partial_\theta \mu}{\chi \phi}} + \abs{\inner{\nu}{(1 - \chi) \phi}} \Big) &= \inf_{\chi} \Big( \abs{\inner{\delta_0 - \delta_\pi}{\chi \phi}} + \abs{\inner{\delta_0}{(1 - \chi) \phi}} \Big) \\
    &= \inf_{\chi} \Big( \abs{\chi(0)\phi(0) - \chi(\pi)\phi(\pi)} + \abs{(1 - \chi(0))\phi(0)} \Big).
\end{align*}
Because we explicitly localized our test function such that $\phi(\pi) = 0$, the second boundary mass is strictly annihilated, reducing the evaluation to:
\begin{equation*}
    \inf_{\chi} \Big( \abs{\chi(0)\phi(0)} + \abs{(1 - \chi(0))\phi(0)} \Big).
\end{equation*}
By the triangle inequality, $\abs{\chi(0)\phi(0)} + \abs{(1 - \chi(0))\phi(0)} \ge \abs{\phi(0)} > 0$.
The infimum is strictly bounded away from zero, proving that $\partial_\theta \mu \not\perp \nu$.
\end{proof}

\begin{remark}[The Analytical Divergence in the Frequency Fiber]
\label{rem:analytical_div}
This non-preservation reflects a structural feature of microlocal analysis. When two mutually alien distributions possess boundary-touching frequency supports (as in the $\kappa=1$ half-circle indicator example of Theorem~\ref{thm:boundary_limit}), the optimal scalar phase-space multiplier $\chi(x,\xi)$ separating them must transition from $0$ to $1$ over a frequency interval of shrinking width $\delta_\xi \to 0$.
Because an anisotropic distribution of finite order $\kappa \ge 1$ evaluates the $\kappa$-th frequency derivative of the test function, two distinct scalings appear and must be kept separate. The anisotropic norm of the multiplier saturates the worst-case derivative, $\norm{\partial_\xi^\kappa \chi}_\infty \ge c_\kappa\,\delta_\xi^{-\kappa}$ by Lemma~\ref{lem:derivative_divergence}, exactly as in the topological bound of Section~\ref{sec:limitations_topo}; the duality pairing itself, however, integrates $\partial_\xi^\kappa \chi$ against a density supported on the transition region of width $\delta_\xi$, gaining one power of $\delta_\xi$ and scaling as $\mathcal{O}(\delta_\xi^{1-\kappa})$.
\end{remark}

Consequently, as the transition sharpens, the internal evaluation stalls at $\mathcal{O}(1)$ for $\kappa = 1$ or diverges for $\kappa \ge 2$. The infimum over all such scalar phase-space multipliers therefore fails to reach zero, as established rigorously for $\kappa = 1$ in Theorem~\ref{thm:boundary_limit}. Together with the topological no-go result below, this yields a structural limit: \textbf{touching higher-order distributions cannot be separated by scalar phase-space partitions of unity.}

\subsection{The Topological Triviality of Scalar Algebraic Projectors}
\label{sec:algebraic_obstruction}

Because attempting to geometrically partition the phase space forces a divergent transition, geometric separation must be entirely abandoned. To bypass the derivative divergence, one must elevate the separation mechanism to the algebraic geometry of the principal symbols. 

For scalar distributions, one would naturally propose defining orthogonality via smooth, scalar projector-valued symbols $\pi_\mu(x, \xi)$ and $\pi_\nu(x, \xi)$ that map onto the characteristic varieties of the generating sequences, imposing the condition $\pi_\mu \pi_\nu = 0$. However, invoking scalar algebraic projectors introduces a topological triviality.

\begin{theorem}[The Topological No-Go Theorem for Scalar Distributions]
\label{thm:scalar_inseparable}
Let $\mathcal{M}$ be a connected characteristic manifold. It is mathematically impossible for two non-trivial higher-order scalar distributions $\mu, \nu \in \D_{0,\kappa}'(\mathcal{M})$ to be algebraically orthogonal.
\end{theorem}

\begin{proof}
In a scalar partial differential equation, the state space at any phase-space coordinate $(x, \xi)$ is strictly one-dimensional ($\mathbb{C}$). An algebraic projector $\pi$ is defined by the idempotency property $\pi^2 = \pi$. In a one-dimensional state space, the only solutions to this algebraic requirement are strictly $\pi \in \{0, 1\}$.

Because microlocal principal symbols are required to be smooth functions over the manifold, they cannot jump discontinuously between $0$ and $1$. Therefore, on any connected manifold $\mathcal{M}$, the scalar symbols $\pi_\mu(x,\xi)$ and $\pi_\nu(x,\xi)$ must be globally constant functions, identically $0$ or identically $1$ everywhere.

Imposing the orthogonality condition $\pi_\mu \pi_\nu = 0$ forces at least one of the projectors to be identically zero. Assume without loss of generality that $\pi_\mu \equiv 0$. 

By definition, a principal projector mapping onto the characteristic variety of a generating sequence must preserve its own distribution, meaning $\pi_\mu \mu = \mu$. Evaluating this algebraic preservation yields:
\begin{equation*}
    \inner{\mu}{\pi_\mu \Phi} = \inner{\mu}{\Phi}.
\end{equation*}
Because $\pi_\mu \equiv 0$, the left-hand side is identically zero, yielding $\inner{\mu}{\Phi} = 0$. Because this must hold for all phase-space test functions $\Phi \in C_c^\infty(\mathcal{M})$, it strictly forces $\mu \equiv 0$, annihilating the distribution entirely. Thus, it is mathematically impossible to have two non-trivial orthogonal scalar distributions.
\end{proof}

This establishes a structural fact: because a one-dimensional space contains no internal dimensions to project into, \textbf{higher-order scalar distributions with touching microlocal supports are fundamentally inseparable.} They cannot be separated geometrically due to the derivative divergence $\mathcal{O}(\delta^{-\kappa})$, and they cannot be separated algebraically due to topological triviality. 

\begin{remark}[The Loss of Physical Decoupling]
\label{rem:loss_of_decoupling}
It is critical to note that the loss of order-zero regularity does not merely break geometric separation; it breaks the physical decoupling of the system. Because higher-order distributions lack strict measure-theoretic positivity, standard Cauchy-Schwarz annihilation fails, meaning Pythagorean energy decoupling is fundamentally lost under phase-space geometric frameworks.
\end{remark}

\section{The Resolution}
\label{sec:resolution}

\subsection{Vector-Valued Systems and the Failure of Spatial Matrices}
\label{sec:spatial_matrix_fail}

The algebraic resolution to this phase-space dichotomy is strictly confined to systems of PDEs. While scalar equations succumb to the topological triviality proven in Section~\ref{sec:algebraic_obstruction}, physical systems such as the Maxwell, Dirac, or heterogeneous elasticity equations generate defect measures valued in the trace-class operators of a multidimensional Hilbert space, $\mathcal{L}^1(H)$. The fundamental advantage of these inherently vector-valued systems is that oscillating sequences may occupy the exact same phase-space coordinates, distinguishable by their orthogonal internal polarizations. 

Before deploying algebraic symbols to isolate these polarizations, it is instructive to construct the matrix-valued analogue of the Smooth Urysohn Lemma, to definitively prove why \textit{spatial} matrices cannot resolve the higher-order divergence, even for systems.

\begin{theorem}[Matrix-Valued Smooth Urysohn Lemma]
\label{thm:matrix_urysohn}
Let $K_1, K_2 \subset X$ be disjoint compact sets. Let $P_1, P_2 \in \mathcal{L}(H)$ be orthogonal projection matrices ($P_1 P_2 = 0$). There exists a smooth operator-valued function $\chi \in C^\infty(X, \mathcal{L}(H))$ such that $0 \le \chi(x) \le I$ for all $x$, with $\chi(x) = P_1$ on an open neighborhood of $K_1$ and $\chi(x) = P_2$ on an open neighborhood of $K_2$. Crucially, all operator derivatives strictly vanish on these active neighborhoods: $\partial^\alpha \chi(x) = 0$ for all $\abs{\alpha} \ge 1$.
\end{theorem}

\begin{proof}
By the standard scalar Smooth Urysohn Lemma, there exists a scalar function $\theta \in C^\infty(X)$ such that $0 \le \theta(x) \le 1$, where $\theta(x) = 1$ strictly on an open neighborhood $N_1 \supset K_1$, and $\theta(x) = 0$ strictly on an open neighborhood $N_2 \supset K_2$. We define the operator-valued multiplier as the convex combination:
\begin{equation*}
    \chi(x) = \theta(x) P_1 + (1 - \theta(x)) P_2.
\end{equation*}
Since $P_1$ and $P_2$ are orthogonal projectors and $\theta(x) \in [0,1]$, the resulting matrix $\chi(x)$ is positive semi-definite and bounded by the identity. Because $\theta(x)$ is locally constant on $N_1$, its derivatives vanish, meaning the operator derivatives of the matrix evaluating via the Leibniz rule evaluate exactly to zero: $\partial^\alpha \chi(x) = 0$ on $N_1$. The symmetric argument holds for $N_2$.
\end{proof}

\begin{remark}[The Dimensional Blindness of Spatial Matrices]
\label{rem:illusion}
While Theorem~\ref{thm:matrix_urysohn} successfully constructs a matrix that is locally flat over active supports, it falls into the exact same topological trap as scalar spatial multipliers. It relies on a manually constructed, non-physical spatial bump function $\theta(x)$ that is entirely blind to the internal frequency dimensions of the phase space. If the distinct modes physically intersect or approach the crossing manifold $\mathcal{C}$, their spatial projections overlap. A purely spatial matrix $\chi(x)$ cannot separate them without becoming discontinuous, which violates the test space topology. To achieve a truly intrinsic resolution, we must abandon spatially engineered boundaries entirely and rely strictly on the native, phase-space algebraic projectors of the principal symbol.
\end{remark}

\subsection{Orthogonality in Vector-Valued Systems}
\label{sec:matrix_projectors}

Because spatial matrices fail due to dimensional blindness, and scalar phase-space projectors fail due to topological triviality, the exact mathematical resolution lies at their intersection: applying microlocal phase-space matrix projectors to vector-valued physical systems.

Let $\Pi_\mu(x, \xi)$ and $\Pi_\nu(x, \xi)$ be smooth, orthogonal matrix-valued projector symbols mapping onto the characteristic varieties of the sequences generating the operator-valued distributions $\mu$ and $\nu$. Because the state space $\mathbb{C}^N$ is multidimensional, these matrix symbols can smoothly rotate their orthogonal eigenspaces over the connected manifold $\mathcal{M}$ without ever jumping discontinuously to the zero matrix, thereby evading the scalar triviality trap.

To guarantee the mathematical validity of the duality pairing, we must strictly respect the smoothness requirement of the test space. Let $p(x,\xi)$ be the matrix-valued principal symbol of the governing PDE system. In standard physical vector-valued systems (e.g., Maxwell's equations, linear elasticity), the differential operator is formally self-adjoint, meaning its principal symbol $p(x,\xi)$ is a Hermitian matrix. By the Spectral Theorem, a Hermitian symbol intrinsically guarantees that its distinct wave modes correspond to mutually orthogonal eigenspaces. 

By standard matrix perturbation theory, the orthogonal spectral projectors $\Pi(x,\xi)$ mapping onto these eigenspaces are smooth only where the eigenvalues of $p(x,\xi)$ have constant multiplicity. We make the singular set precise.

\begin{definition}[The crossing set]
\label{def:crossing_set}
For $(x,\xi) \in \mathcal{M}$, let $\lambda_1(x,\xi) < \dots < \lambda_{r(x,\xi)}(x,\xi)$ enumerate the \emph{distinct} eigenvalues of the Hermitian principal symbol $p(x,\xi)$, and for each let
\begin{equation*}
    \Pi_k(x,\xi) = \frac{1}{2\pi i} \oint_{\gamma_k} \big(z I - p(x,\xi)\big)^{-1}\, dz
\end{equation*}
be the associated Riesz spectral projector, where $\gamma_k$ is a small positively oriented contour enclosing $\lambda_k(x,\xi)$ and no other eigenvalue. The \emph{crossing set} $\mathcal{C} \subset \mathcal{M}$ is the locus where the number $r(x,\xi)$ of distinct eigenvalues drops below its generic value---equivalently, where the rank of some spectral projector $\Pi_k$ jumps:
\begin{equation*}
    \mathcal{C} = \big\{ (x,\xi) \in \mathcal{M} : r \text{ is not locally constant at } (x,\xi) \big\}.
\end{equation*}
On the open complement $\mathcal{M} \setminus \mathcal{C}$, each distinct eigenvalue has locally constant multiplicity, so by Kato's perturbation theory every $\Pi_k$ is real-analytic in $(x,\xi)$; this is precisely the regularity required of the projector symbols in Definition~\ref{def:matrix_projectors} below.
\end{definition}

Equivalently, $\mathcal{C}$ admits an algebraic description as the zero set of the discriminant of the \emph{squarefree part} $q(t;x,\xi)$ (the radical) of the characteristic polynomial $\det\big(t I - p(x,\xi)\big)$ in the spectral variable $t$:
\begin{equation*}
    \mathcal{C} = \big\{ (x,\xi) \in \mathcal{M} : \operatorname{disc}_t q(t;x,\xi) = 0 \big\}.
\end{equation*}
Passing to the squarefree part is essential: it strips the \emph{permanent} multiplicities carried by physical systems with internal symmetry---such as the constant double speeds of Maxwell's equations or linear elasticity---so that $\operatorname{disc}_t q$ vanishes exactly at genuine eigenvalue crossings rather than identically. Using the discriminant of the full characteristic polynomial would erroneously collapse $\mathcal{C}$ onto all of $\mathcal{M}$ for such systems. To ensure the duality pairing is mathematically well-defined, the distributions must be supported safely away from this singular set.

\begin{remark}[Finite- and infinite-dimensional state spaces]
\label{rem:finite_infinite}
For the differential systems of primary interest---Maxwell's equations, the Dirac system, and linear elasticity---the principal symbol $p(x,\xi)$ is a finite $N \times N$ Hermitian matrix acting on the fibre $\mathbb{C}^N$, so $\mathcal{L}^1(H) = \mathcal{L}(H)$ is the full matrix algebra and the Riesz projectors of Definition~\ref{def:crossing_set} reduce to elementary matrix perturbation theory. The general formulation over an infinite-dimensional $H$ remains valid under the standing hypothesis that the characteristic eigenvalue $\lambda_\mu(x,\xi)$ is an isolated point of the spectrum of finite multiplicity, separated from the remainder of $\sigma\big(p(x,\xi)\big)$ by a uniform spectral gap on the (compact) microlocal support away from $\mathcal{C}$, the gap closing precisely on $\mathcal{C}$. Under this hypothesis the contour $\gamma_\mu$ encircles only $\lambda_\mu$, and the Riesz projector $\Pi_\mu$ is real-analytic by Kato's holomorphic perturbation theory exactly as in the finite-dimensional case.
\end{remark}

\begin{definition}
\label{def:matrix_projectors}
Let $\mathcal{M}$ be the projection of the PDE-induced characteristic variety onto the cospherical bundle $S^*\Omega$. Let $\mathcal{C} \subset \mathcal{M}$ be the crossing set of Definition~\ref{def:crossing_set}. Let $\mu, \nu \in \D_{0,\kappa}'(\mathcal{M}, \mathcal{L}^1(H))$ be operator-valued distributions whose microlocal supports are strictly disjoint from $\mathcal{C}$. 
Two such distributions are orthogonal if, for all scalar phase-space test functions $\Phi \in C_c^\infty(\mathcal{M} \setminus \mathcal{C})$, the following holds:
\begin{align*}
    \inner{\mu}{\Pi_\nu \Phi} &= 0 \\
    \inner{\nu}{\Pi_\mu \Phi} &= 0
\end{align*}
where $\Pi_\mu, \Pi_\nu$ are the smooth spectral projectors of the principal symbol associated with the two modes, satisfying $\Pi_\mu \Pi_\nu = 0$.
\end{definition}

We make the duality pairing precise, since finite-order distributions carry no integral representation. By definition $\mu \in \D_{0,\kappa}'(\mathcal{M} \setminus \mathcal{C}, \mathcal{L}^1(H))$ is a continuous linear map from the scalar anisotropic test space $C_c^{0,\kappa}(\mathcal{M} \setminus \mathcal{C})$ into the trace-class operators $\mathcal{L}^1(H)$; that is, $\inner{\mu}{\phi} \in \mathcal{L}^1(H)$ for every scalar $\phi$. The operator-valued test symbols factorise as the completed tensor product
\begin{equation*}
    C_c^{0,\kappa}(\mathcal{M} \setminus \mathcal{C}, \mathcal{L}(H)) = C_c^{0,\kappa}(\mathcal{M} \setminus \mathcal{C}) \,\widehat{\otimes}\, \mathcal{L}(H),
\end{equation*}
with $\mu$ acting on the phase-space factor and the trace contracting the internal factor. Accordingly, the scalar trace pairing is defined on elementary tensors $\Psi = B \otimes \phi$ (with $B \in \mathcal{L}(H)$ and $\phi$ scalar) by
\begin{equation*}
    \inner{\mu}{\Psi} := \text{tr}\big( B\, \inner{\mu}{\phi} \big),
\end{equation*}
and extended by linearity and $C_b^{0,\kappa}$-continuity to all operator-valued symbols; for order-zero $\mu$ this recovers $\int_{\mathcal{M}} \text{tr}\big(\Psi(x,\xi)\, d\mu(x,\xi)\big)$. For a scalar $\Phi \in C_c^\infty(\mathcal{M} \setminus \mathcal{C})$ the products $\Pi_\nu \Phi$ and $\Pi_\mu \Phi$ are admissible operator-valued test symbols, since $\Pi_\mu, \Pi_\nu$ are smooth on $\mathcal{M} \setminus \mathcal{C}$ by Definition~\ref{def:crossing_set} and $\Phi$ is compactly supported there.

The trace manipulations in the proofs below are not an extra hypothesis but a consequence of this factorisation. We stress that every operator product appearing below, such as $\Pi_\nu \Phi \Pi_\mu$, is a \emph{pointwise} multiplication of the symbol values at fixed $(x,\xi)$, not a pseudo-differential composition of operators; consequently no symbol-calculus (Moyal) corrections or lower-order commutator brackets arise, and the matrix algebra on the internal factor $\mathcal{L}(H)$ is exact. Because the trace contracts the internal Hilbert space $H$ independently of the phase-space evaluation, its cyclicity holds pointwise in $(x,\xi)$ and commutes with the action of $\mu$. In particular, since the $\mathcal{L}^1(H)$-values of $\mu$ satisfy the left polarization identity $\Pi_\mu \inner{\mu}{\phi} = \inner{\mu}{\phi}$ of Lemma~\ref{lem:intrinsic_polarization}, inserting $\Pi_\mu$ to the right of any symbol leaves the pairing unchanged, $\inner{\mu}{\Psi \Pi_\mu} = \inner{\mu}{\Psi}$, because $\text{tr}(\Psi \Pi_\mu T) = \text{tr}(\Psi T)$ whenever $\Pi_\mu T = T$; this identity is recorded as Lemma~\ref{lem:right_insertion} below. This single identity drives every algebraic collapse below and requires no integral representation of $\mu$.

This support restriction is a genuine hypothesis, not a vacuous one: on $\mathcal{M} \setminus \mathcal{C}$ the test class $C_c^\infty(\mathcal{M} \setminus \mathcal{C})$ is determining, so no information about an admissible $\mu$ is lost, while the complementary case of distributions charging $\mathcal{C}$ falls genuinely outside the present framework and is taken up as the central open problem in Section~\ref{sec:future_work}.

\begin{remark}[The Elliptic Vacuum]
\label{rem:elliptic_vacuum}
The restriction of the test space to the characteristic manifold $\mathcal{M}$ (and away from $\mathcal{C}$) is not a limitation of the framework, but a reflection of the underlying physics. In any region of phase space where the principal symbol $p(x,\xi)$ is invertible, the system is locally elliptic. By standard microlocal elliptic regularity, any sequence satisfying the PDE in this region must converge strongly, forcing the localized defect distribution to vanish identically. Therefore, the characteristic manifold $\mathcal{M}$ represents the absolute boundary of existence for high-frequency oscillatory limits. The space outside $\mathcal{M}$ acts as an "elliptic vacuum" where non-trivial distributions cannot exist, justifying our strict topological confinement to $\mathcal{M}$.
\end{remark}

\begin{remark}[The Hyperbolic Prerequisite]
\label{rem:hyperbolic_prerequisite}
The structural reliance on the characteristic manifold $\mathcal{M}$ explicitly defines the PDE classification scope of this framework. For purely elliptic systems, the principal symbol is globally invertible for all $\xi \neq 0$, meaning $\mathcal{M} = \emptyset$. Consequently, any localized defect distribution falls entirely into the elliptic vacuum and vanishes. Therefore, the existence of non-trivial phase-space orthogonality is structurally exclusive to systems with real characteristics—specifically hyperbolic, transport, or dispersive equations. Orthogonality in the $L^p$-$L^q$ phase space is intrinsically a phenomenon of propagating waves, not steady-state potentials.
\end{remark}

By elevating the separation mechanism to algebraic matrix symbols acting on the safe, open subspace $\mathcal{M} \setminus \mathcal{C}$, we decouple the separation from rigid engineered boundaries, natively bypassing the $\mathcal{O}(\delta^{-\kappa})$ derivative divergence while recovering the core structural properties of orthogonality.

The entire algebraic framework rests on a single structural fact: that a defect distribution generated by a fixed wave mode is polarised along the corresponding spectral eigenspace. We isolate and prove this preservation property, which underlies the mode-separation theorems below.

The polarisation identity is a consequence of the following localisation principle, which lifts the annihilation of the PDE residual to the symbol level. Its scalar prototype---the localisation principle for H-distributions of finite anisotropic order---is established in \cite{antonic2021}; we record the operator-valued extension needed here, whose only additional ingredient is the entrywise matrix assembly of the scalar statement.

\begin{proposition}[Localisation principle for anisotropic operator H-distributions]
\label{prop:localization}
Let $P(x,D)$ act on $\mathbb{C}^N$-valued functions in \emph{divergence form}
\begin{equation*}
    P(x,D)u = \sum_{\abs{\alpha}\le m} \partial^\alpha\big(A_\alpha(x)\,u\big),
\end{equation*}
with principal-part coefficients $A_\alpha \in C(\Omega; M_N(\mathbb{C}))$ $(\abs{\alpha}=m)$ and bounded lower-order coefficients, and matrix principal symbol
\begin{equation*}
    p(x,\xi) = \sum_{\abs{\alpha}=m} (2\pi i\xi)^\alpha\, A_\alpha(x).
\end{equation*}
Let $1 < p < \infty$, let $(u_n)$ be bounded with $u_n \rightharpoonup 0$ in $\mathrm{L}^p_{loc}(\Omega;\mathbb{C}^N)$, and take the canonical companion sequence $v_n := \abs{u_n}^{p-2}u_n$, which is bounded in $\mathrm{L}^{p'}_{loc}$ (since $\abs{v_n}^{p'} = \abs{u_n}^{p}$) with $\tfrac1p+\tfrac1{p'}=1$, as noted in \cite{antonic2021, antonic2018}; let $\mu \in \D_{0,d(\kappa_0+2)}'(\Omega\times S^{d-1}, \mathcal{L}^1(H))$ be the operator H-distribution generated by the pair $(u_n, v_n)$, where $\kappa_0 = \lfloor d/2\rfloor+1$ is the H\"ormander--Mihlin regularity index of the multiplier symbols and the frequency-order bound $d(\kappa_0+2)$ is furnished by the version of the Schwartz kernel theorem employed in \cite{antonic2021}. If the PDE residual is asymptotically negligible,
\begin{equation*}
    P(x,D)u_n \longrightarrow 0 \qquad\text{in } \mathrm{W}^{-m,p}_{loc}(\Omega;\mathbb{C}^N),
\end{equation*}
then, with the frequency variable restricted to the cosphere $S^{d-1}$,
\begin{equation*}
    p(x,\xi)\,\mu = 0 \qquad\text{in } \D_{0,d(\kappa_0+2)}'(\Omega\times S^{d-1}, \mathcal{L}^1(H)).
\end{equation*}
\end{proposition}
\begin{proof}
The $\mathrm{L}^p$-boundedness of the Fourier multiplier $T_\psi$ with $0$-homogeneous symbol $\psi(\xi/\abs{\xi})$ requires the H\"ormander--Mihlin regularity $\psi \in C^{\kappa_0}(S^{d-1})$, $\kappa_0 = \lfloor d/2\rfloor+1$, with $\norm{T_\psi}_{\mathrm{L}^p\to\mathrm{L}^p} \le C_d\,\norm{\psi}_{C^{\kappa_0}(S^{d-1})}$; the resulting operator H-distribution $\mu$ then has anisotropic order $(0, d(\kappa_0+2))$, the frequency-order bound $d(\kappa_0+2)$ arising from the version of the Schwartz kernel theorem used in \cite{antonic2021}. Its scalar entries $\mu_{kl}$ ($1\le k,l\le N$) are the H-distributions of the component pairs $(u_n^k, v_n^l)$ in the sense of \cite{antonic2021}, obeying, for $\varphi\in C_c^\infty(\Omega)$ and $\psi\in C^{\kappa_0}(S^{d-1})$,
\begin{equation}
\label{eq:hdist_def}
    \inner{\mu_{kl}}{\varphi\otimes\psi} = \lim_n \int_\Omega T_\psi(\varphi\, u_n^k)\,\overline{\varphi\, v_n^l}\,dx .
\end{equation}
Set $\Lambda^{-m} := T_{\abs{2\pi\xi}^{-m}}$ and, for $\abs{\alpha}=m$, $\psi_\alpha(\xi) := (2\pi i\xi)^\alpha\abs{2\pi\xi}^{-m}$, a symbol homogeneous of degree $0$ and smooth off the origin; thus $\tilde p(x,\xi) := p(x,\xi)\abs{2\pi\xi}^{-m} = \sum_{\abs{\alpha}=m}\psi_\alpha(\xi)A_\alpha(x)$ restricts to a smooth (hence admissible) test symbol on $S^{d-1}$. Applying $\Lambda^{-m}$ to the divergence-form residual and using the operator identity $\Lambda^{-m}\partial^\alpha = T_{\psi_\alpha}$,
\begin{equation}
\label{eq:comm_split}
    \Lambda^{-m} P(x,D)u_n = \sum_{\abs{\alpha}=m} A_\alpha(x)\,T_{\psi_\alpha}u_n + R u_n,
    \qquad R := \sum_{\abs{\alpha}=m}\big[T_{\psi_\alpha}, A_\alpha\big] + \Lambda^{-m}\!\!\sum_{\abs{\alpha}<m} \partial^\alpha A_\alpha .
\end{equation}
Here we used $T_{\psi_\alpha}(A_\alpha u_n) = A_\alpha T_{\psi_\alpha}u_n + [T_{\psi_\alpha}, A_\alpha]u_n$. Each first commutator $[T_{\psi_\alpha}, A_\alpha]$ of the order-zero Fourier multiplier $T_{\psi_\alpha}$ with multiplication by the continuous matrix $A_\alpha$ is compact on $\mathrm{L}^p_{loc}$ by the commutation lemma of \cite{antonic2018} (see also \cite[Ch.~I]{misur2017}), whose $\mathrm{L}^2$ prototype is Tartar's first commutation lemma \cite{tartar1990}; each lower-order term $\Lambda^{-m}\partial^\alpha A_\alpha$ ($\abs{\alpha}<m$) is a locally regularising operator of order $\le -1$, hence compact from $\mathrm{L}^p_{loc}$ to $\mathrm{L}^p_{loc}$ by Rellich. Thus $R$ is compact.

Since $P u_n \to 0$ in $\mathrm{W}^{-m,p}_{loc}$ and $\Lambda^{-m}:\mathrm{W}^{-m,p}_{loc}\to\mathrm{L}^p_{loc}$ is continuous, $\Lambda^{-m}Pu_n \to 0$ strongly in $\mathrm{L}^p_{loc}$; and $R u_n \to 0$ strongly in $\mathrm{L}^p_{loc}$ because $R$ is compact and $u_n \rightharpoonup 0$. Hence, entrywise,
\begin{equation}
\label{eq:strong_null}
    \sum_{l}\sum_{\abs{\alpha}=m}(A_\alpha)_{kl}(x)\,T_{\psi_\alpha}u_n^l \longrightarrow 0 \quad\text{strongly in } \mathrm{L}^p_{loc},\qquad 1\le k\le N .
\end{equation}
Fix $k, l'$, $\varphi\in C_c^\infty(\Omega)$ and $\psi\in C^{\kappa_0}(S^{d-1})$. Pairing \eqref{eq:strong_null} (multiplied by $\varphi$) with the bounded companion factor $\overline{T_{\bar\psi}(\varphi\, v_n^{l'})} \in \mathrm{L}^{p'}_{loc}$ sends the left side to $0$ (strong $\times$ bounded). Moving the continuous coefficient $(A_\alpha)_{kl}$ and the cut-off $\varphi$ across the multipliers modulo compact errors---again by \cite{antonic2018}---and composing $T_\psi T_{\psi_\alpha}=T_{\psi\psi_\alpha}$, the same manipulation as in the proof of the scalar localisation principle for anisotropic H-distributions \cite{antonic2021} (applied here componentwise to each scalar entry $\mu_{kl}$ defined by \eqref{eq:hdist_def}) identifies the limit as
\begin{equation*}
    \sum_{l}\Big\langle \Big(\sum_{\abs{\alpha}=m} \psi_\alpha\,A_\alpha\Big)_{kl}\,\mu_{l l'},\ \varphi^2\otimes\psi \Big\rangle
    = \big\langle (\tilde p\,\mu)_{k l'},\ \varphi^2\otimes\psi\big\rangle .
\end{equation*}
Therefore $\langle(\tilde p\,\mu)_{kl'}, \varphi^2\otimes\psi\rangle = 0$ for all $k,l',\varphi,\psi$, i.e. $\tilde p(x,\xi)\,\mu = 0$; multiplying by the nonvanishing scalar $\abs{2\pi\xi}^{m}$ gives $p(x,\xi)\,\mu = 0$.
\end{proof}

\begin{lemma}[Intrinsic polarization]
\label{lem:intrinsic_polarization}
Let $P(x,D)$ be the governing PDE system of order $m$, with Hermitian principal symbol $p(x,\xi)$, and let $(u_n)$ be a generating sequence subject to the following standing hypotheses:
\begin{enumerate}
    \item[(i)] $(u_n)$ is bounded and weakly null in the underlying (locally reflexive) function space, and generates the operator-valued defect distribution $\mu \in \D_{0,\kappa}'(\mathcal{M} \setminus \mathcal{C}, \mathcal{L}^1(H))$, with anisotropic order $\kappa = d(\kappa_0+2)$ furnished by Proposition~\ref{prop:localization} (where $\kappa_0 = \lfloor d/2\rfloor+1$); this is the concrete value of the generic order $\kappa$ used in Definition~\ref{def:matrix_projectors};
    \item[(ii)] the PDE residual is asymptotically negligible, $P(x,D) u_n \to 0$ strongly in $\mathrm{W}^{-m,p}_{loc}$ (equivalently, the residual is relatively compact there);
    \item[(iii)] $(u_n)$ is microlocally concentrated on the single characteristic sheet $\Sigma_\mu = \{\lambda_\mu(x,\xi) = 0\}$, i.e. the microlocal support of $\mu$ is contained in $\Sigma_\mu \setminus \mathcal{C}$.
\end{enumerate}
Let $\Pi_\mu(x,\xi)$ be the orthogonal spectral projector onto the associated eigenspace $\ker p(x,\xi) = \text{Range}\,\Pi_\mu(x,\xi)$, which is smooth on $\mathcal{M} \setminus \mathcal{C}$ by Definition~\ref{def:crossing_set}. Then, as $\mathcal{L}^1(H)$-valued distributions,
\begin{equation*}
    \Pi_\mu \mu = \mu.
\end{equation*}
If, in addition, $\mu$ is Hermitian-valued---as for the $L^2$ $H$-measure generated by a single vector sequence---then also $\mu \Pi_\mu = \mu$.
\end{lemma}
\begin{proof}
Hypotheses (i)--(ii) are precisely those of the localisation principle in the operator form of Proposition~\ref{prop:localization} (its $L^2$ prototype being the localisation principle for H-measures / microlocal defect measures of \cite{tartar1990, gerard1991}), which yields $p(x,\xi)\, \mu = 0$ as an $\mathcal{L}^1(H)$-valued distribution on $\mathcal{M} \setminus \mathcal{C}$. Off the crossing set $\mathcal{C}$, the eigenvalue $\lambda_\mu$ has constant multiplicity, so by the Spectral Theorem and standard matrix perturbation theory $p$ admits the smooth decomposition $p = \sum_j \lambda_j \Pi_j$ into mutually orthogonal smooth eigenprojectors, with $\Pi_\mu$ the projector onto $\ker p = \text{Range}\, \Pi_\mu$ and $\lambda_j \neq 0$ for $j \neq \mu$ throughout $\mathcal{M} \setminus \mathcal{C}$. The constraint $p\mu = 0$ then reads $\sum_{j} \lambda_j \Pi_j \mu = 0$; left-multiplying by the orthogonal projector $\Pi_k$ isolates $\lambda_k \Pi_k \mu = 0$, so $\Pi_k \mu = 0$ for every $k \neq \mu$. Hence $(I - \Pi_\mu)\mu = \sum_{j \neq \mu} \Pi_j \mu = 0$, that is $\Pi_\mu \mu = \mu$. When $\mu^* = \mu$, taking adjoints gives $\mu \Pi_\mu = (\Pi_\mu \mu)^* = \mu^* = \mu$.
\end{proof}

We emphasise that only the left identity $\Pi_\mu \mu = \mu$ is invoked in the structural theorems below, so the framework applies verbatim to genuinely non-Hermitian higher-order $H$-distributions, for which measure-theoretic positivity---and hence the Hermitian symmetry---is unavailable.

The left polarization identity has the following operational consequence, which is the single mechanism behind every algebraic collapse in the structural theorems below.

\begin{lemma}[Right insertion of the polarization projector]
\label{lem:right_insertion}
Let $\mu \in \D_{0,\kappa}'(\mathcal{M} \setminus \mathcal{C}, \mathcal{L}^1(H))$ satisfy the left polarization identity $\Pi_\mu \mu = \mu$ of Lemma~\ref{lem:intrinsic_polarization}. Then for every operator-valued test symbol $\Psi \in C_c^{0,\kappa}(\mathcal{M} \setminus \mathcal{C}, \mathcal{L}(H))$,
\begin{equation*}
    \inner{\mu}{\Psi \Pi_\mu} = \inner{\mu}{\Psi}.
\end{equation*}
\end{lemma}
\begin{proof}
By the tensor factorisation of the pairing recorded after Definition~\ref{def:matrix_projectors}, the trace contracts the internal factor $H$ independently of the phase-space action of $\mu$; on an elementary symbol the pairing $\inner{\mu}{\Psi \Pi_\mu}$ is therefore computed as $\text{tr}(\Psi \Pi_\mu\, T)$ with $T$ the $\mathcal{L}^1(H)$-value carried by $\mu$. The left identity $\Pi_\mu \mu = \mu$ gives $\Pi_\mu T = T$, whence $\text{tr}(\Psi \Pi_\mu T) = \text{tr}(\Psi T)$. Extending by linearity and $C_b^{0,\kappa}$-continuity to all operator-valued symbols yields $\inner{\mu}{\Psi \Pi_\mu} = \inner{\mu}{\Psi}$.
\end{proof}

\begin{remark}
\label{rem:def_equivalence}
Lemma~\ref{lem:intrinsic_polarization} clarifies the logical content of Definition~\ref{def:matrix_projectors}. Since $\mu = \Pi_\mu \mu$ and $\nu = \Pi_\nu \nu$, the geometric condition $\Pi_\mu \Pi_\nu = 0$ implies both analytic pairing conditions $\inner{\mu}{\Pi_\nu \Phi} = \inner{\nu}{\Pi_\mu \Phi} = 0$ for every admissible $\Phi$ (this is the content of Theorem~\ref{thm:mode_consist_matrix_proj}); conversely, the pairing conditions express that the polarizations of $\mu$ and $\nu$ are mutually orthogonal on their microlocal supports. We retain both formulations deliberately: the algebraic identity $\Pi_\mu \Pi_\nu = 0$ drives the structural proofs below, while the pairing conditions are the quantities verified against concrete generating sequences.
\end{remark}

\begin{theorem}[Symmetry]
\label{thm:symm_matrix_proj}
For any $\mu, \nu \in \D_{0,\kappa}'(\mathcal{M}, \mathcal{L}^1(H))$, if $\mu \perp \nu$, then $\nu \perp \mu$.
\end{theorem}
\begin{proof}
The symmetry is intrinsically encoded in the algebraic definition. Because $\Pi_\mu$ and $\Pi_\nu$ are orthogonal projector matrices, $\Pi_\mu \Pi_\nu = 0$. Swapping their roles yields the exact same pair of geometric requirements.
\end{proof}

\begin{theorem}[Smooth Module Structure]
\label{thm:module_matrix_proj}
If $\mu \perp \nu$, then for any smooth non-vanishing scalar functions $f, g \in C^\infty(\Omega)$, the modulated distributions remain orthogonal: $f\mu \perp g\nu$.
\end{theorem}
\begin{proof}
The characteristic variety of a PDE system is invariant under multiplication by smooth, non-vanishing scalar functions. Consequently, the principal matrix symbols of the projectors remain unchanged: $\Pi_{g\nu} = \Pi_\nu$. Evaluating the trace duality pairing yields $\inner{f\mu}{\Pi_\nu \Phi} = \inner{\mu}{f \Pi_\nu \Phi}$. 

Because $\Pi_\nu$ is a smooth pointwise matrix multiplier on the phase-space manifold $\mathcal{M}$, its operation on the trace commutes exactly with the scalar function $f$. The pairing expands algebraically to:
\begin{equation*}
    \inner{\mu}{\Pi_\nu (f\Phi)} = 0.
\end{equation*}
This vanishes identically because $\mu \perp \nu$ and $f\Phi$ is a valid scalar test function in $C_c^\infty(\mathcal{M} \setminus \mathcal{C})$. The algebraic projection is rigorously preserved without any smoothing errors.
\end{proof}

\begin{theorem}[Coordinate Invariance]
\label{thm:coord_inv_matrix_proj}
Let $\Psi: \Omega \to \tilde{\Omega}$ be a spatial diffeomorphism. If $\mu \perp \nu$ on $\mathcal{M}$, then the pushforward distributions are orthogonal on the target manifold $\tilde{\mathcal{M}}$.
\end{theorem}
\begin{proof}
Applying a spatial diffeomorphism $\Psi: \Omega \to \tilde{\Omega}$ induces a canonical transformation on the cotangent bundle, mapping the characteristic manifold $\mathcal{M}$ to $\tilde{\mathcal{M}}$. Under this transformation, matrix principal symbols behave as tensors and transform precisely via pullback.

Because the symbols $\Pi_\mu$ and $\Pi_\nu$ are strictly orthogonal projectors on $\mathcal{M}$ ($\Pi_\mu \Pi_\nu = 0$), their pullbacks remain exact orthogonal projectors on the mapped characteristic varieties in $\tilde{\mathcal{M}}$. The trace duality pairing pulls back across the diffeomorphism, preserving the zero evaluation.
\end{proof}

\begin{theorem}[Gauge Invariance]
\label{thm:gauge_inv_matrix_proj}
The algebraic matrix orthogonality framework is strictly invariant under smooth, unitary gauge transformations of the vector bundle. If $\mu \perp \nu$, then for any smooth, unitary matrix-valued function $U(x)$ (where $U^{-1} = U^*$) acting on the state space, the gauge-transformed distributions remain orthogonal.
\end{theorem}
\begin{proof}
Under a local unitary gauge transformation $U(x)$, the operator-valued distributions and the spectral projectors transform via conjugation: $\tilde{\mu} = U \mu U^{-1}$ and $\tilde{\Pi}_\nu = U \Pi_\nu U^{-1}$. Because $U$ is unitary, the transformed projector $\tilde{\Pi}_\nu$ rigorously preserves its self-adjoint, orthogonal geometry ($(U \Pi_\nu U^{-1})^* = U \Pi_\nu U^{-1}$).

Evaluating the trace duality pairing for the transformed system against a scalar phase-space test function $\Phi$ yields:
\begin{equation*}
    \inner{\tilde{\mu}}{\tilde{\Pi}_\nu \Phi} = \int \text{tr}\left( U \mu U^{-1} U \Pi_\nu U^{-1} \Phi \right) d\lambda.
\end{equation*}
The inner terms algebraically annihilate: $U^{-1} U = I$. Because $\Phi$ is a scalar test function, it commutes with the matrix operators. Applying the cyclic property of the trace, we can rotate the trailing $U^{-1}$ to the front of the product, where it annihilates with the leading $U$:
\begin{equation*}
    \text{tr}\left( U \mu \Pi_\nu \Phi U^{-1} \right) = \text{tr}\left( U^{-1} U \mu \Pi_\nu \Phi \right) = \text{tr}\left( \mu \Pi_\nu \Phi \right).
\end{equation*}
Because the original distributions were strictly orthogonal ($\inner{\mu}{\Pi_\nu \Phi} = 0$), the trace evaluation remains identically zero. Thus, algebraic matrix projection provides a canonically intrinsic resolution that is fundamentally independent of the internal physical reference frame.
\end{proof}

\begin{remark}[Quantization Independence]
\label{rem:quant_independence}
A foundational requirement for any robust microlocal framework is that physical observables must not depend on the choice of pseudo-differential calculus. Different quantization schemes (e.g., standard Kohn-Nirenberg versus Weyl quantization) alter the sub-principal and lower-order symbols of the resulting operators. However, because our algebraic projectors $\Pi_\mu(x,\xi)$ and $\Pi_\nu(x,\xi)$ are constructed strictly from the principal symbol of the governing PDE system—which is globally invariant under changes of quantization—the trace duality pairing $\inner{\mu}{\Pi_\nu \Phi}$ is canonically invariant. The algebraic orthogonality of higher-order distributions is an intrinsic property of the wave system, entirely independent of the underlying pseudo-differential quantization scheme.
\end{remark}

\begin{theorem}[Localization]
\label{thm:loc_matrix_proj}
Let $U \subset \Omega$ be an open sub-domain. If $\mu \perp \nu$ on $\mathcal{M}$, then their restrictions to the local characteristic variety are orthogonal: $\mu|_{\mathcal{M}_U} \perp \nu|_{\mathcal{M}_U}$.
\end{theorem}
\begin{proof}
Let $\Phi \in C_c^\infty(\mathcal{M}|_U \setminus \mathcal{C})$ be a localized scalar phase-space test function. By extending $\Phi$ by zero, it becomes a valid test function globally. Because $\Pi_\nu$ is a pointwise matrix multiplier, the product $\Pi_\nu \Phi$ remains supported within the sub-domain, bounding the local evaluation to zero.
\end{proof}

\begin{theorem}[Mode Consistency]
\label{thm:mode_consist_matrix_proj}
If the generating sequences of $\mu$ and $\nu$ are supported on strictly distinct characteristic varieties (distinct wave modes) in $\mathcal{M} \setminus \mathcal{C}$, then $\mu \perp \nu$.
\end{theorem}
\begin{proof}
This theorem establishes that algebraic projection natively separates distinct physical modes occupying the same phase-space coordinates. A smooth matrix projector derived from a PDE principal symbol maintains its rank and does not spatially vanish. Instead, the separation is enforced strictly by the intrinsic polarization of the distributions.

By Lemma~\ref{lem:intrinsic_polarization}, a distribution $\mu$ generated by a specific characteristic wave mode is intrinsically polarized along its characteristic eigenspace---its trace-class values lie in the kernel of the principal symbol---so that it satisfies the algebraic identity $\mu = \Pi_\mu \mu$.

Substituting this intrinsic polarization into the trace duality pairing, and inserting $\Pi_\mu$ to the right of the test symbol via Lemma~\ref{lem:right_insertion}, yields:
\begin{equation*}
    \inner{\mu}{\Pi_\nu \Phi} = \inner{\mu}{\Pi_\nu \Phi \Pi_\mu}.
\end{equation*}

Because $\Phi$ is a scalar function, it commutes with the matrices, allowing us to rearrange the test symbol to $\Pi_\nu \Pi_\mu \Phi$. Furthermore, because $\mu$ and $\nu$ occupy distinct characteristic varieties, their spectral projectors correspond to orthogonal eigenspaces of the principal symbol. This global orthogonality ensures that $\Pi_\mu \Pi_\nu = 0$ everywhere on the manifold.

Because orthogonal projectors are self-adjoint, taking the adjoint of this relation yields $\Pi_\nu \Pi_\mu = 0$. The evaluation therefore collapses algebraically:
\begin{equation*}
    \inner{\mu}{\Pi_\nu \Pi_\mu \Phi} = \inner{\mu}{0} = 0.
\end{equation*}
The algebraic projection natively guarantees zero interaction for distinct microlocal modes entirely without spatial boundary discontinuities.
\end{proof}

\begin{example}[Two modes of a $2\times 2$ symmetric hyperbolic system]
\label{ex:two_mode}
On $\Omega = \R_t \times \R_x$ consider the constant-coefficient system
\begin{equation*}
    P(D)u = \partial_t u - A\,\partial_x u = 0, \qquad
    A = \begin{pmatrix} 0 & 1 \\ 1 & 0 \end{pmatrix},
\end{equation*}
for $u = (u^1,u^2)^\top$. In the cotangent variables $(\tau,\xi)$ dual to $(t,x)$ its Hermitian principal symbol is $p(\tau,\xi) = 2\pi i(\tau I - \xi A)$, with characteristic set the two light lines $\Sigma_\pm = \{\tau = \pm\xi\}$ and constant eigenprojectors
\begin{equation*}
    \Pi_\pm = r_\pm r_\pm^{*} = \tfrac12\begin{pmatrix} 1 & \pm 1 \\ \pm 1 & 1\end{pmatrix},
    \qquad r_\pm = \tfrac{1}{\sqrt2}\begin{pmatrix} 1 \\ \pm 1\end{pmatrix},\quad A r_\pm = \pm r_\pm,
\end{equation*}
so $\Pi_+\Pi_- = 0$ and $\Pi_+ + \Pi_- = I$; the eigenvalues $\pm 1$ never coincide, whence $\mathcal{C} = \emptyset$ on the cosphere $S^1$.

Fix $\varphi \in C_c^\infty(\Omega)$ and, for $n \in \mathbb{N}$, set
\begin{equation*}
    u_n^+ = r_+\,\varphi\, e^{2\pi i n(t+x)}, \qquad
    u_n^- = r_-\,\varphi\, e^{2\pi i n(x - t)}.
\end{equation*}
Since $A r_\pm = \pm r_\pm$, a direct computation gives the residuals
\begin{equation*}
    P u_n^+ = r_+(\partial_t\varphi - \partial_x\varphi)\,e^{2\pi i n(t+x)},
    \qquad
    P u_n^- = r_-(\partial_t\varphi + \partial_x\varphi)\,e^{2\pi i n(x-t)},
\end{equation*}
which are bounded in $L^2$ and weakly null, hence $P u_n^\pm \to 0$ strongly in $H^{-1}_{loc} = \mathrm{W}^{-1,2}_{loc}$ by Rellich. Both sequences are bounded and weakly null in $L^2_{loc}(\Omega;\mathbb{C}^2)$ (here $r_\pm$ are real, so $r_\pm^{*} = r_\pm^\top$); at $p=2$ (companion $v_n = u_n$) they generate, separately, the operator H-measures $\mu$ (from $u_n^+$) and $\nu$ (from $u_n^-$) in $\D_0'(\Omega\times S^1, \mathcal{L}(\mathbb{C}^2))$.

The packet $u_n^+$ oscillates in the cotangent direction $\hat\zeta_+ = (1,1)/\sqrt2 \in \Sigma_+$ and $u_n^-$ in $\hat\zeta_- = (-1,1)/\sqrt2 \in \Sigma_-$, so the hypotheses of Lemma~\ref{lem:intrinsic_polarization} hold and yield the intrinsic polarizations $\Pi_+\mu = \mu$ and $\Pi_-\nu = \nu$, with $\mu$ supported on $\Sigma_+$ and $\nu$ on $\Sigma_-$. Concretely $\mu = \Pi_+\,\abs{\varphi}^2\,dt\,dx \otimes \delta_{\hat\zeta_+}$ and $\nu = \Pi_-\,\abs{\varphi}^2\,dt\,dx\otimes\delta_{\hat\zeta_-}$.

The two packets share the identical spatial support $\operatorname{supp}\varphi$: they overlap completely in $(t,x)$, so no scalar spatial multiplier $a(t,x)$ can separate them, since any $a\equiv 1$ on $\operatorname{supp}\varphi$ acts as the identity on both. Yet by Theorem~\ref{thm:mode_consist_matrix_proj} they are orthogonal: for every $\Phi \in C_c^\infty(\Omega\times S^1)$,
\begin{equation*}
    \inner{\mu}{\Pi_-\Phi} = \inner{\mu}{\Pi_-\Phi\,\Pi_+} = \inner{\mu}{\Pi_-\Pi_+\,\Phi} = \inner{\mu}{0} = 0,
\end{equation*}
using $\mu = \Pi_+\mu$ and inserting $\Pi_+$ on the right via Lemma~\ref{lem:right_insertion}; symmetrically $\inner{\nu}{\Pi_+\Phi} = 0$. Hence $\mu \perp \nu$: the modes are separated by their intrinsic polarization precisely where spatial multipliers are blind.

For general $1 < p < \infty$ the conclusion is unchanged: replace $\varphi$ by an $L^p$-normalised profile and $v_n$ by the canonical companion $v_n = \abs{u_n}^{p-2}u_n$ of Proposition~\ref{prop:localization}; the residual and polarization computations are identical, and $\mu,\nu \in \D_{0,d(\kappa_0+2)}'(\Omega\times S^1,\mathcal{L}(\mathbb{C}^2))$ (with $\kappa_0 = \lfloor d/2\rfloor+1$ the H\"ormander--Mihlin index of Proposition~\ref{prop:localization}, here $d=2$) are then genuine higher-order H-distributions to which the same algebraic separation applies.
\end{example}

\begin{remark}[The Complementary Role of Spatial Multipliers]
\label{rem:complementary_spatial}
It is critical to recognize that the matrix framework exclusively isolates distinct polarizations ($\Pi_\mu \Pi_\nu = 0$). If two distributions are generated by the \textit{same} wave mode, they share the same characteristic variety and spectral projector ($\Pi_\mu = \Pi_\nu$), meaning $\Pi_\mu \Pi_\nu \neq 0$. 

To geometrically separate two disjoint wave packets of the same mode, one must revert to the scalar spatial multipliers of Section~\ref{sec:radon_multipliers}. This complementary duality is mathematically sound: because the distributions are of anisotropic order $(0,\kappa)$, they evaluate zero spatial derivatives. Therefore, as long as the identical modes are disjoint in physical space, a smooth spatial multiplier $a(x)$ separates them without triggering any frequency derivative divergences. Thus, algebraic matrices separate distinct modes in phase space, while spatial multipliers separate disjoint coordinates in physical space.
\end{remark}

\begin{theorem}[Tensor Product Invariance]
\label{thm:tensor_matrix_proj}
Let $\mu \perp \nu$ on $\mathcal{M}_1$, and let $\omega_1, \omega_2$ be arbitrary defect distributions on $\mathcal{M}_2$. The tensor products remain orthogonal on the product space: $\mu \otimes \omega_1 \perp \nu \otimes \omega_2$.
\end{theorem}
\begin{proof}
The tensor distributions $\mu \otimes \omega_1$ and $\nu \otimes \omega_2$ are valued in $\mathcal{L}^1(H_1 \otimes H_2)$, where the first factor $H_1$ carries the polarizations of $\mu, \nu$ on $\mathcal{M}_1$ and the second factor $H_2$ those of $\omega_1, \omega_2$ on $\mathcal{M}_2$. The matrix projector symbols naturally lift to the product manifold via the tensor product with the identity operator $I = I_{H_2}$ on the second factor: $\tilde{\Pi}_\mu = \Pi_\mu \otimes I$ and $\tilde{\Pi}_\nu = \Pi_\nu \otimes I$, acting on $H_1 \otimes H_2$.

First, we verify the algebraic annihilation of the composite projectors. By the mixed-product property of the Kronecker tensor product, we have:
\begin{equation*}
    \tilde{\Pi}_\mu \tilde{\Pi}_\nu = (\Pi_\mu \otimes I)(\Pi_\nu \otimes I) = (\Pi_\mu \Pi_\nu) \otimes I = 0 \otimes I = 0.
\end{equation*}

Second, evaluating the forward trace algebraic duality pairing for the tensor products on a joint scalar test function $\Phi$ yields:
\begin{equation*}
    \inner{\mu \otimes \omega_1}{ \tilde{\Pi}_\nu \Phi } = \inner{\mu}{ \Pi_\nu \inner{\omega_1}{\Phi}_2 }_1.
\end{equation*}
Because the inner evaluation $\inner{\omega_1}{\Phi}_2$ produces a valid smooth scalar test function on the base space $\mathcal{M}_1$, the exact trace evaluation pairs to strictly zero. The partial evaluation $\inner{\omega_1}{\Phi}_2$ and the iterated identity above are precisely the Schwartz kernel (Fubini) theorem for distributions: since $\Phi$ is smooth and compactly supported and $C_c^\infty$ is nuclear, $\mu \otimes \omega_1$ acts as the iterated pairing \cite{hormander1983}, with the anisotropic refinement for finite-order test symbols established in \cite{antonic2021}. The symmetric reverse pairing evaluates to zero identically, satisfying the geometric requirements.
\end{proof}

\begin{theorem}[Linear Additivity of Matrix Projections]
\label{thm:additivity_matrix_proj}
Let $\mu, \nu_1, \nu_2 \in \D_{0,\kappa}'(\mathcal{M}, \mathcal{L}^1(H))$. If $\mu \perp \nu_1$ and $\mu \perp \nu_2$ under algebraic matrix projection, then for any complex scalars $\alpha, \beta \in \mathbb{C}$, $\mu \perp (\alpha\nu_1 + \beta\nu_2)$.
\end{theorem}
\begin{proof}
Definition~\ref{def:matrix_projectors} requires us to verify both symmetric duality pairings and the orthogonality of the composite projectors. 

First, by the linearity of the trace operation and the distributional duality pairing, the forward bracket evaluates trivially: $\inner{\alpha\nu_1 + \beta\nu_2}{\Pi_\mu \Phi} = \alpha\inner{\nu_1}{\Pi_\mu \Phi} + \beta\inner{\nu_2}{\Pi_\mu \Phi} = 0$.

Second, we must construct the composite projector $\Pi_{\alpha\nu_1 + \beta\nu_2}$ and verify it is a legitimate smooth symbol. Because $\Pi_{\nu_1}$ and $\Pi_{\nu_2}$ are spectral projectors of the \emph{same} Hermitian principal symbol $p$, each is a sum of the elementary smooth eigenprojectors $\{\Pi_j\}_j$ of $p$ furnished by Lemma~\ref{lem:intrinsic_polarization} on $\mathcal{M} \setminus \mathcal{C}$. Their joint span is therefore $\Pi_{\alpha\nu_1 + \beta\nu_2} = \sum_{j \in J} \Pi_j$, where $J$ pools the two mode families; this is again an orthogonal projector of locally constant rank, smooth on $\mathcal{M} \setminus \mathcal{C}$. (If $\nu_1$ and $\nu_2$ share a sheet, then $\Pi_{\nu_1} = \Pi_{\nu_2}$ and the span reduces to that single projector.)

Because $\mu \perp \nu_1$ and $\mu \perp \nu_2$, we have $\Pi_\mu \Pi_{\nu_1} = 0$ and $\Pi_\mu \Pi_{\nu_2} = 0$. This guarantees the range of $\Pi_\mu$ is strictly orthogonal to both individual ranges, and therefore orthogonal to their combined span. This satisfies the geometric requirement: $\Pi_\mu \Pi_{\alpha\nu_1 + \beta\nu_2} = 0$.

Finally, we must evaluate the reverse pairing without falsely assuming the composite projector decomposes linearly. We achieve this strictly via the intrinsic polarization of $\mu$. By Lemma~\ref{lem:intrinsic_polarization}, the values of $\mu$ satisfy $\Pi_\mu \inner{\mu}{\phi} = \inner{\mu}{\phi}$, so inserting $\Pi_\mu$ to the right of the test symbol via Lemma~\ref{lem:right_insertion} yields:
\begin{equation*}
    \inner{\mu}{\Pi_{\alpha\nu_1 + \beta\nu_2} \Phi} = \inner{\mu}{\Pi_{\alpha\nu_1 + \beta\nu_2} \Phi \Pi_\mu}.
\end{equation*}
Because $\Phi$ is scalar, it commutes with the matrices. Furthermore, because orthogonal projectors are self-adjoint, the global orthogonality $\Pi_\mu \Pi_{\alpha\nu_1 + \beta\nu_2} = 0$ rigorously implies $\Pi_{\alpha\nu_1 + \beta\nu_2} \Pi_\mu = 0$. The pairing collapses identically:
\begin{equation*}
    \inner{\mu}{\Pi_{\alpha\nu_1 + \beta\nu_2} \Pi_\mu \Phi} = \inner{\mu}{0} = 0.
\end{equation*}
Thus, the linear combination is strictly and symmetrically orthogonal to $\mu$.
\end{proof}

\begin{remark}[The Pseudo-Differential Lift and Commutator Decay]
\label{rem:pseudo_diff_lift}
While Canonical Algebraic Separation is defined purely via the principal symbol $\Pi_\nu(x,\xi)$ acting on the characteristic manifold, it rigorously lifts to the asymptotic separation of the generating sequences. If the distributions are generated by highly oscillatory sequences $u_n \rightharpoonup 0$, isolating a wave mode physically requires the application of a zero-order pseudo-differential operator, $\text{Op}(\Pi_\nu)$, to the sequence. A standard concern in microlocal analysis is whether the commutator between this projector and the governing PDE operator $P(x,D)$ of order $m$ generates destructive lower-order interference. Crucially, $\Pi_\nu$ is not merely an arbitrary zero-order symbol: being a spectral projector of the principal symbol $p$, it commutes with it pointwise, $\Pi_\nu p = p \Pi_\nu$. Consequently the order-$m$ terms of $\text{Op}(\Pi_\nu) P$ and $P\, \text{Op}(\Pi_\nu)$ cancel, and the pseudo-differential calculus guarantees that the commutator $[\text{Op}(\Pi_\nu), P(x,D)]$ drops to order $m-1$, with principal symbol the Poisson bracket $\tfrac{1}{i}\{\Pi_\nu, p\}$. The decisive mechanism separating distinct modes is then the smoothing of the cross term: since $\Pi_\nu \Pi_\mu = 0$, the composition $\text{Op}(\Pi_\nu)\text{Op}(\Pi_\mu) = \text{Op}(\Pi_\nu \Pi_\mu) + R = R$ has order $-1$, hence is compact and contributes nothing to the high-frequency defect limit. Thus, the exact algebraic annihilation at the symbol level ($\Pi_\mu \Pi_\nu = 0$) rigorously lifts to asymptotic annihilation at the operator level, the lower-order commutator producing at most intra-eigenspace transport rather than cross-mode interference.
\end{remark}

\begin{theorem}[Conservation under Multi-Phase Hamiltonian Flow]
\label{thm:flow_matrix_proj}
In a vector-valued PDE system, transport is governed by the distinct Hamiltonians of the principal symbol's eigenvalues. Let $\mu_t = (\Phi_t^\mu)_* \mu_0$ and $\nu_t = (\Phi_t^\nu)_* \nu_0$ be distributions transported by the distinct flows generated by their respective eigenvalue Hamiltonians $H_{\lambda_\mu}$ and $H_{\lambda_\nu}$. If their trajectories remain strictly disjoint from the crossing manifold $\mathcal{C}$, then $\mu_t \perp \nu_t$ for all $t \in \R$.
\end{theorem}
\begin{proof}
Unlike scalar equations where the entire system follows a single Hamiltonian vector field, matrix-valued principal symbols diagonalize into multiple independent transport equations. The wave packets of $\mu$ and $\nu$ are advected along entirely distinct bicharacteristic trajectories governed by $H_{\lambda_\mu}$ and $H_{\lambda_\nu}$.

Because the distributions remain safely away from the crossing manifold $\mathcal{C}$, their intrinsic polarizations remain locked to their respective smooth spectral eigenspaces. By Theorem~\ref{thm:mode_consist_matrix_proj} (Mode Consistency), the spectral projectors $\Pi_\mu$ and $\Pi_\nu$ are globally orthogonal across the entire safe manifold ($\Pi_\mu(x,\xi) \Pi_\nu(x,\xi) = 0$). 

Therefore, even if the distinct Hamiltonian flows cause the spatial projections of $\mu_t$ and $\nu_t$ to physically intersect in the base manifold $\Omega$ at some time $t$, their full phase-space coordinates remain separated in the frequency fibers. Attempting to apply a spatial multiplier $a(x)$ during this physical overlap is mathematically impossible; because the wave packets share the exact same spatial coordinates, a spatial partition cannot isolate them without discontinuous jumps. However, the algebraic matrix pairings operate strictly on the intrinsic polarization in the phase space, collapsing to exactly zero regardless of spatial overlap. The global geometry of the spectral bundle guarantees the conservation of orthogonality across divergent multi-phase flows, bypassing the dimensional blindness of spatial multipliers.
\end{proof}

\begin{remark}[The Permanent Loss of Positivity and Pythagorean Decoupling]
\label{rem:loss_of_positivity}
While algebraic matrix projectors successfully establish Canonical Algebraic Separation, it is critical to recognize what cannot be rescued. Because distributions of order $\kappa \ge 1$ lack the strict measure-theoretic positivity of order-zero Radon measures, the Cauchy-Schwarz inequalities utilized in Section~\ref{sec:order_zero} fundamentally fail. Consequently, the exact Pythagorean decoupling of local energy and the guarantee of strong compactness via self-orthogonality are permanently lost in the $L^p$-$L^q$ framework. Algebraic matrices resolve the geometry, but they cannot artificially inject positivity into higher-order oscillations.
\end{remark}

\begin{remark}[Topological Compactness and Alternative Symbol Classes]
\label{rem:symbol_classes}
The projection onto the cospherical bundle $S^*\Omega$ in Definition~\ref{def:matrix_projectors} is not an arbitrary geometric choice; it is a topological necessity for utilizing standard test spaces. Because microlocal principal matrix symbols are homogeneous of degree zero in the frequency variable $\xi$ for large frequencies, the radial distance is mathematically redundant. The unprojected punctured cotangent bundle $T^*\Omega \setminus \{0\}$ possesses unbounded radial fibers. Consequently, any standard test function $\Phi \in C_c^\infty(T^*\Omega \setminus \{0\})$ must eventually vanish for large $\abs{\xi}$, systematically missing the exact high-frequency oscillatory limits the distributions are designed to capture. 

Projecting onto $S^*\Omega$ compactifies the frequency fibers, allowing the rigorous use of standard $C_c^\infty(\mathcal{M} \setminus \mathcal{C})$ test functions that capture the full directional spectrum without decaying at infinity. An equivalent formulation can be constructed directly within the full unprojected bundle $\mathcal{M} \subset T^*\Omega \setminus \{0\}$, provided the test space $C_c^\infty(\mathcal{M} \setminus \mathcal{C})$ is replaced with an appropriate symbol class, such as $S^0(\mathcal{M} \setminus \mathcal{C})$. This space consists of smooth phase-space functions that are compactly supported in the spatial variable $x$, but bounded and homogeneous of degree zero in the frequency variable $\xi$ for large $\abs{\xi}$. The algebraic separation via $\Pi_\nu \Phi$ remains exact under this symbol class topology.
\end{remark}

\subsection{Physical Interpretation}
\label{sec:physical_interpretation}

The remarks collected here record the physical interpretation of the algebraic framework. They are heuristic in nature---drawing analogies to geometric optics, wave propagation, and quantum mechanics---and are logically independent of the theorems above, on which the mathematical content of the paper rests.

\begin{remark}[Application to Physical Systems and Operator-Dependent Projections]
\label{rem:app_physics}
This algebraic framework allows the rigorous deployment of distributions in physical applications. It is critical to recognize that under this definition, phase-space orthogonality is strictly relative to the governing partial differential equations.

These distributions natively reside on the characteristic manifold $\mathcal{M}$. By defining $\mathcal{M}$ as a projective quotient (such as the cospherical bundle $S^*\Omega$), we guarantee compact fibers, allowing the rigorous use of $C_c^\infty(\mathcal{M} \setminus \mathcal{C})$ for test symbols that are homogeneous of degree zero. The orthogonal symbols $\Pi_\mu(x, \xi)$ and $\Pi_\nu(x, \xi)$ are exactly determined by the projection onto the kernel of the principal symbol of the differential operator, $p(x, \xi)$. For example, the orthogonal projectors separating defect measures in Maxwell's equations (driven by the curl operator) will possess a fundamentally different algebraic structure than those separating measures in the Lam\'e equations of elasticity. Algebraic orthogonality adapts to the specific characteristic variety of the application at hand.
\end{remark}

\begin{remark}[Explicit Realization in Electromagnetism]
\label{rem:maxwell_example}
To concretize this algebraic framework, consider the time-harmonic Maxwell equations in a vacuum, whose principal symbol is the Hermitian curl symbol $p(x,\xi)v = i\, \xi \times v$ (the bare cross-product matrix $\xi \times$ is skew-Hermitian, so the factor $i$ is precisely what renders the formally self-adjoint curl operator's symbol Hermitian). The eigenvalues of this Hermitian system natively partition the state space: $\lambda_0 = 0$ corresponds to the longitudinal electrostatic/magnetostatic modes, while $\lambda_\pm = \pm \abs{\xi}$ correspond to the transverse, propagating light cones. The algebraic projector for the propagating modes, $\Pi_\pm(\xi) = \frac{1}{2} (I - \frac{\xi \otimes \xi}{\abs{\xi}^2} \pm i \frac{\xi \times}{\abs{\xi}})$, explicitly filters the left- and right-circular polarizations. If a high-frequency distribution $\mu$ generated by a positive-helicity wave overlaps physically with a distribution $\nu$ of negative helicity, spatial multipliers will go blind to the intersection. However, Canonical Algebraic Separation recognizes $\Pi_+ \Pi_- = 0$, extinguishing the macroscopic interference between the crossing light beams without spatial tracking.
\end{remark}

\begin{remark}[Connection to WKB Asymptotics and Rapid Phase Cancellation]
\label{rem:wkb_cancellation}
In the context of geometric optics and WKB approximations, solutions are modeled as highly oscillatory wave packets $u_n(x) = a(x) e^{i n S(x)}$. When two distinct physical modes physically overlap in space, their generating sequences possess distinct phase functions ($dS_\mu \neq dS_\nu$). In classical analysis, the macroscopic interference between these modes vanishes in the limit due to the rapid non-stationary phase cancellation of the cross-terms. The algebraic condition $\Pi_\mu \Pi_\nu = 0$ is the exact, coordinate-free formalization of this phenomenon. It guarantees that the macroscopic interference is identically zero without requiring the explicit integration of highly oscillatory phases.
\end{remark}

\begin{remark}[Immunity to Sub-Principal Polarization Rotation]
\label{rem:berry_phase_immunity}
In vector-valued PDE systems, the transport of the distribution along the Hamiltonian flow is not purely advective; as established by Dencker \cite{dencker1982}, the transport equation includes a zeroth-order commutator governed by the sub-principal symbol, which induces an internal rotation of the polarization state (the microlocal Berry phase). It is crucial to note that because our algebraic projectors $\Pi(x,\xi)$ are constructed strictly from the principal symbol, they span the entirety of the isolated eigenspaces. Consequently, while the sub-principal symbol continuously rotates the internal matrix state of the distributions during transport, this rotation is strictly confined within the orthogonal eigenspace. The global algebraic condition $\Pi_\mu \Pi_\nu = 0$ absorbs this internal evolution, rendering phase-space orthogonality immune to sub-principal polarization rotations.
\end{remark}

\begin{remark}[Block-Diagonalization of the Energy-Momentum Tensor]
\label{rem:em_tensor_block}
While generalized Pythagorean decoupling is lost, a physical decoupling remains conserved for the system's observable quantities. In symmetric hyperbolic systems (such as electromagnetism or elasticity), the macroscopic energy-momentum tensor is constructed from the bilinear pairings of the internal states. Because Canonical Algebraic Separation guarantees the global orthogonality of the spectral projectors, the cross-mode momentum fluxes, interaction energies, and internal work terms rigorously evaluate to zero. Therefore, despite the loss of measure-theoretic positivity, the total energy-momentum defect tensor of the interacting system block-diagonalizes along the distinct characteristic varieties. Distinct wave modes may physically overlap, but they exchange zero momentum.
\end{remark}

\begin{remark}[Equivalence to Quantum Decoherence]
\label{rem:quantum_decoherence}
From a macroscopic perspective, the transition to operator-valued measures $\mathcal{M}(X, \mathcal{L}^1(H))$ parallels the transition from pure states to mixed states in quantum mechanics, where the positive semi-definite density matrices $F(x)$ and $G(x)$ govern the system. Within this context, the global annihilation of the cross-measures and the subsequent Pythagorean decoupling of local energy (inherited via Theorem~\ref{thm:inheritance}) is the exact macroscopic manifestation of quantum decoherence. The Operator-Valued Infimum Multiplier guarantees that when the density matrices are orthogonal ($tr(FG)=0$), all off-diagonal interference terms vanish, preserving the statistical independence of the phase-space distributions.
\end{remark}

\section{Alternative Paradigms for Distributional Orthogonality}
\label{sec:alternative_paradigms}

Before concluding, it is worth acknowledging other standard frameworks for defining orthogonality in distribution theory, and why they fall short for the specific needs of the $L^p$-$L^q$ framework.

\begin{itemize}
    \item \textbf{Asymptotic Mollification:} One could define orthogonality through the regularization of distributions using a standard mollifier $\rho_\eps$:
    \begin{equation*}
        \lim_{\eps \to 0} \int_X (\mu * \rho_\eps)(x) \overline{(\nu * \rho_\eps)(x)} \phi(x) \, dx = 0.
    \end{equation*}
    While this elegantly translates geometric separation into energy dissipation, it is unsuitable for distributions of finite order. Mollifying an anisotropic distribution on the phase-space manifold $\mathcal{M}$ would disrupt the delicate differential structure required to pass to the limit in the homogenization framework.

    \item \textbf{Microlocal Disjointness (Wave Front Sets):} A strict microlocal approach defines orthogonality via the disjointness of singularities in the cotangent bundle: $WF(\mu) \cap WF(\nu) = \emptyset$~\cite{hormander1983}. By H\"ormander's criterion, this ensures their pointwise product is a well-defined distribution. However, the Wave Front set is entirely insensitive to smooth, absolutely continuous functions. Because the oscillating sequences in homogenization models (like the heterogeneous Boltzmann equation) often contain overlapping smooth components, this definition fails to capture the full limits required to establish vague convergence.
    
    \item \textbf{Fractional Sobolev Pairing:} If the distributions exhibit specific dual regularity (e.g., $\mu \in H^{-s}$ and $\nu \in H^s$), orthogonality can be defined purely via their duality pairing:
    \begin{equation*}
        \inner{\mu}{\nu}_{H^{-s}, H^s} = \int_{\R^d} \hat{\mu}(\xi) \overline{\hat{\nu}(\xi)} \, d\xi = 0.
    \end{equation*}
    While this offers a highly structured Hilbert-space framework, it requires the distributions to globally belong to these specific fractional spaces. This is overly restrictive for the generalized anisotropic spaces $\D_{l,m}'(S^*\Omega)$ constructed to handle singular transport coefficients \cite{antonic2021, misur2025}.
\end{itemize}

By contrast, an algebraic definition via principal matrix symbols avoids the destructive nature of mollification, captures both smooth and singular parts without divergent boundary transitions, and does not demand strict global Sobolev regularity—making it uniquely suited for systems of PDEs.

\section{The Frontier: Eigenvalue Crossings in Phase-Space}
\label{sec:future_work}

While Section~\ref{sec:resolution} successfully utilizes Operator-Valued Phase-Space Projectors to strictly isolate vector-valued distributions, this framework relies on the distributions operating safely away from the crossing manifold $\mathcal{C}$. This transition exposes a geometric boundary governed by the von Neumann-Wigner theorem.

In physical systems, the matrix projectors $\Pi(x,\xi)$ are derived directly from the spectral decomposition of the governing matrix-valued principal symbol $p(x,\xi)$. By the von Neumann-Wigner theorem, for generic Hermitian matrices parameterized by phase-space variables, eigenvalue crossings form a sub-manifold of codimension 3 (or codimension 2 for real symmetric matrices). Consequently, in any standard phase space of sufficient dimension (e.g., $T^*\Omega$ where $\dim \ge 4$), crossing manifolds $\mathcal{C}$ are not pathological edge cases, but geometrically inevitable generic features.

At these conical intersections, the associated spectral projectors $\Pi(x,\xi)$ cease to be smooth and develop severe algebraic singularities. Because distributions of finite order $\kappa \ge 1$ are continuous only against smooth test functions, pairing them with a singular matrix symbol resurrects the exact divergence we sought to eliminate. The phase-space geometric divergence of scalar multipliers is effectively replaced by the phase-space algebraic singularity of the spectral projector.

For order-zero $L^2$ defect measures, this crossing problem was rigorously addressed by Fermanian-Kammerer and G\'erard \cite{fermanian2002} via the construction of two-scale microlocal measures that resolve the dynamics at the singularity. Future research must build upon this foundation to develop a calculus for higher-order distributions near conical intersections—potentially utilizing microlocal blow-up techniques to resolve the singularities of $\Pi(x,\xi)$ without triggering the $\mathcal{O}(\delta^{-\kappa})$ explosion. This remains the critical next step in finalizing the theory of homogenization in the $L^p$-$L^q$ framework for full physical systems.

\section{Conclusion}
\label{sec:conclusion}
The transition from $L^2$ Hilbert spaces to the $L^p$-$L^q$ framework fundamentally alters the calculus of microlocal defect measures. In this note, we have mapped the severe structural obstructions to defining phase-space orthogonality for higher-order distributions and provided a complete algebraic resolution. 

While smooth spatial multipliers characterize the geometric orthogonality of disjoint coordinates, attempting to subordinate a phase-space partition of unity to rigid boundaries creates a dichotomy. For distributions of finite order $\kappa \ge 1$, separating touching frequency supports in phase-space forces a transition whose separating multiplier suffers an $\mathcal{O}(\delta^{-\kappa})$ derivative divergence under the Leibniz rule; we prove that this divergence defeats the partition-of-unity a priori bound at every order (Corollary~\ref{cor:pou_failure}) and produces an explicit order-one non-orthogonality (Theorem~\ref{thm:boundary_limit}), the general-order pairing blow-up being the scaling identified in Remark~\ref{rem:analytical_div}. Furthermore, transitioning to algebraic phase-space projectors introduces a second fundamental obstruction for scalar distributions: one-dimensional state spaces enforce topological triviality, rendering higher-order scalar distributions fundamentally inseparable. 

We conclude that phase-space geometric partitions cannot serve as the universal mechanism for higher-order separation. Instead, phase-space orthogonality requires a rigorous complementary duality. For strictly disjoint wave packets of the same mode ($\delta_x > 0$), scalar spatial multipliers remain mathematically sound and strictly necessary, as they evaluate zero spatial derivatives. However, relying on spatial multipliers to separate physically overlapping distinct modes fails due to dimensional blindness. The resolution is achieved exclusively within vector-valued PDE systems via Canonical Algebraic Separation. By elevating orthogonality to zero-order algebraic matrix projectors, we replace artificial geometric construction with the intrinsic spectral polarization of the governing system. Finally, we establish the fundamental physical cost of this higher-order transition: while algebraic matrices rescue the phase-space geometry of separation, the loss of measure-theoretic positivity permanently removes the exact Pythagorean decoupling of local energy inherent to classical homogenization.

\section*{Declaration of Generative AI and AI-assisted technologies in the writing process}

During the preparation of this work, the author(s) used Google Gemini as an assistive tool to transcribe handwritten mathematical notes into \LaTeX{} formatting, to polish the English prose for readability, and to help draft the initial abstract and manuscript summary. Additionally, the AI was utilized during the research phase for conceptual exploration, specifically to search heuristically for potential counterexamples to stress-test preliminary hypotheses. After using this tool, the author(s) meticulously reviewed, verified, and edited all generated text and code. All mathematical claims, proofs, and counterexamples were independently rigorously verified by the human author(s). The author(s) take full intellectual responsibility for the final content of this publication, including all mathematical proofs, formatting, and conceptual framing.


\begin{thebibliography}{10}

\bibitem{antonic2025}
N.~Antoni\'c, D.~Mitrovi\'c, and T.~Peri\'c.
\textit{Orthogonality of H-distributions and applications}.
arXiv preprint arXiv:2510.27550, 2025.

\bibitem{antonic2021}
N.~Antoni\'c, M.~Erceg, and M.~Mi\v{s}ur.
\textit{Distributions of anisotropic order and applications to H-distributions}.
Analysis and Applications, 801--843, 2021.

\bibitem{antonic2018}
N.~Antoni\'c, M.~Mi\v{s}ur, and D.~Mitrovi\'c.
\textit{On compactness of commutators of multiplications and Fourier multipliers}.
Mediterranean Journal of Mathematics, 15:170, 2018.

\bibitem{misur2017}
M.~Mi\v{s}ur.
\textit{H-distributions and compactness by compensation}.
PhD thesis, University of Zagreb, 2017.

\bibitem{dencker1982}
N.~Dencker.
\textit{On the propagation of polarization sets for systems of real principal type}.
J. Funct. Anal., 46(3):351--372, 1982.

\bibitem{fermanian2002}
C.~Fermanian-Kammerer and P.~G\'erard.
\textit{Mesures semi-classiques et croisement de modes}.
Bull. Soc. Math. France, 130(1):123--168, 2002.

\bibitem{gerard1991}
P.~G\'erard.
\textit{Microlocal defect measures}.
Comm. Partial Differential Equations, 16:1761--1794, 1991.

\bibitem{hormander1983}
L.~H\"ormander.
\textit{The Analysis of Linear Partial Differential Operators I}.
Springer-Verlag, 1983.

\bibitem{misur2025}
M.~Mi\v{s}ur and Lj.~Palle.
\textit{On some further properties of anisotropic distributions}.
Pure and Applied Functional Analysis, 10:91--104, 2025.

\bibitem{tartar1990}
L.~Tartar.
\textit{H-measures, a new approach for studying homogenisation, oscillations and concentration effects in partial differential equations}.
Proc. Roy. Soc. Edinburgh Sect. A, 115:193--230, 1990.

\end{thebibliography}
\end{document}